\documentclass{amsart}
\usepackage{amsthm}
\usepackage{amssymb}
\usepackage{tikz-cd}
\usepackage{colonequals}
\usepackage{enumerate}
\usepackage{eucal}
\usepackage[french,english]{babel}

\usepackage{hyperref}
\hypersetup{%
  bookmarksnumbered=true,%
  colorlinks=true,%
  linkcolor=blue,%
  citecolor=blue,%
  filecolor=blue,%
  menucolor=blue,%
  urlcolor=blue,%
  bookmarksopen=true,%
  bookmarksdepth=2,%
  pageanchor=true}

\makeatletter
\@namedef{subjclassname@2020}{\textup{2020} Mathematics Subject Classification}
\makeatother



\numberwithin{equation}{section}

\swapnumbers

\theoremstyle{plain}
\newtheorem{theorem}[equation]{Theorem}
\newtheorem{proposition}[equation]{Proposition}
\newtheorem{lemma}[equation]{Lemma} 
\newtheorem{corollary}[equation]{Corollary}

\theoremstyle{definition}
\newtheorem{definition}[equation]{Definition}
\newtheorem{example}[equation]{Example}

\newtheorem*{claim}{Claim}

\theoremstyle{remark}
\newtheorem{remark}[equation]{Remark}

\newcommand{\Add}{\operatorname{Add}}
\newcommand{\cat}[1]{\mathcal{#1}}
\newcommand{\Coker}{\operatorname{Coker}}

\newcommand{\colim}{\operatorname{colim}}
\newcommand{\comp}{\mathop{\circ}}
\newcommand{\col}{\colon}

\newcommand{\hocolim}{\operatorname{hocolim}}
\newcommand{\dbcat}[1]{\mathbf{D}^{\mathrm{b}}(\mod #1)}
\newcommand{\dcat}[1]{\mathbf{D}(\Mod #1)}
\newcommand{\dsing}{\mathbf{D}_{\mathrm{sg}}}
\newcommand{\DHom}[1]{\operatorname{Hom}_{\mathbf D(#1)}}
\newcommand{\End}{\operatorname{End}}

\newcommand{\Ext}{\operatorname{Ext}}
\newcommand{\tExt}{\operatorname{\rlap{$\smash{\widehat{\mathrm{Ext}}}$}\phantom{\mathrm{Ext}}}}
\newcommand{\env}[1]{{{#1}^{\mathrm{ev}}}}

\newcommand{\Fin}{\operatorname{Fin}}
\newcommand{\ges}{{\scriptscriptstyle\geqslant}}
\newcommand{\gp}{\operatorname{GP}}
\newcommand{\gi}{\operatorname{GI}}
\newcommand{\Gproj}{\operatorname{Gproj}}
\newcommand{\uGproj}{\underline{\Gproj}}
\newcommand{\GProj}{\operatorname{GProj}}
\newcommand{\uGProj}{\underline{\operatorname{GProj}}}
\newcommand{\GInj}{\operatorname{GInj}}

\newcommand{\hh}[1]{H^{*}(#1)}
\newcommand{\Hom}{\operatorname{Hom}}
\newcommand{\KHom}[1][A]{\operatorname{Hom_{\mathbf K(#1)}}}
\newcommand{\id}{\operatorname{id}}

\newcommand{\Inj}{\operatorname{Inj}}
\newcommand{\ires}{\mathbf{i}}
\newcommand{\ipres}{\mathbf{j}}

\newcommand{\iso}{\xrightarrow{\raisebox{-.4ex}[0ex][0ex]{$\scriptstyle{\sim}$}}}

\newcommand{\longiso}{\xrightarrow{\ \raisebox{-.4ex}[0ex][0ex]{$\scriptstyle{\sim}$}\ }}
\newcommand{\longliso}{\xleftarrow{\ \raisebox{-.4ex}[0ex][0ex]{$\scriptstyle{\sim}$}\ }}

\newcommand{\KacInj}[1]{\mathbf{K}_{\mathrm{ac}}(\Inj #1)}
\newcommand{\kan}{\mathbf{N}}
\newcommand{\Ker}{\operatorname{Ker}}
\newcommand{\KFlat}[1]{\mathbf{K}(\mathrm{Flat} #1)}
\newcommand{\KInj}[1]{\mathbf{K}(\Inj #1)}
\newcommand{\KInjc}[1]{\mathbf{K}^{\mathrm c}(\Inj #1)}
\newcommand{\kinj}[1]{\mathbf{K}_{\mathrm{inj}}(#1)}
\newcommand{\KMod}[1]{\mathbf {K}(\Mod #1)}

\newcommand{\pres}{\mathbf{p}}
\newcommand{\Prj}{\operatorname{Proj}}

\newcommand{\KacProj}[1]{\mathbf{K}_{\mathrm{ac}}(\Prj #1)}
\newcommand{\KProj}[1]{\mathbf{K}(\Prj #1)}
\newcommand{\KProjc}[1]{\mathbf{K}^{\mathrm c}(\Prj #1)}
\newcommand{\kproj}[1]{\mathbf{K}_{\mathrm{proj}}(#1)}

\newcommand{\Loc}{\operatorname{Loc}}
\renewcommand{\mod}{\operatorname{mod}}
\newcommand{\Mod}{\operatorname{Mod}}
\newcommand{\pdim}{\operatorname{proj\,dim}}

\newcommand{\bbN}{\mathbb N} 
\newcommand{\rank}{\operatorname{rank}}
\newcommand{\RHom}{\operatorname{RHom}}
\newcommand{\sHom}{\underline{\Hom}}
\newcommand{\Spec}{\operatorname{Spec}}
\newcommand{\supp}{\operatorname{supp}}

\newcommand{\Thick}{\operatorname{Thick}}
\newcommand{\Tor}{\operatorname{Tor}}

\newcommand{\lotimes}{\otimes^{\mathrm L}}
\newcommand{\op}[1]{{#1}^{\mathrm{op}}}
\newcommand{\lra}{\longrightarrow}
\newcommand{\wh}{\widehat}

\newcommand{\xra}{\xrightarrow}
\newcommand*{\intref}[2]{\def\tmp{#1}\ifx\tmp\empty\hyperref[#2]{\ref*{#2}}\else\hyperref[#2]{#1~\ref*{#2}}\fi}

\newcommand{\mcE}{\mathcal{E}}

\newcommand{\bfa}{\mathbf{a}}
\newcommand{\bff}{\mathbf{f}}
\newcommand{\bfh}{\mathbf{h}}
\newcommand{\bfj}{\mathbf{j}}
\newcommand{\bfg}{\mathbf{g}}
\newcommand{\bfq}{\mathbf{q}}
\newcommand{\bfr}{\mathbf{r}}
\newcommand{\bfs}{\mathbf{s}}
\newcommand{\bft}{\mathbf{t}}

\newcommand{\bfK}{\mathbf K}

\newcommand{\fm}{\mathfrak{m}} 
\newcommand{\fp}{\mathfrak{p}}
\newcommand{\fq}{\mathfrak{q}} 
\newcommand{\eps}{\varepsilon}
\newcommand{\gam}{\varGamma}

\title[Gorenstein algebras]{The Nakayama functor and its completion for \\ Gorenstein algebras}

\author[Iyengar and Krause]{Srikanth  B. Iyengar and Henning Krause}

\address{Srikanth B. Iyengar\\ 
Department of Mathematics\\
University of Utah\\ 
Salt Lake City, UT 84112\\ 
U.S.A.}

\address{Henning Krause\\ 
Fakult\"at f\"ur Mathematik\\ 
Universit\"at Bielefeld\\ 
33501 Bielefeld\\ 
Germany.}

\begin{document}
\dedicatory{To Bill Crawley-Boevey on his 60th birthday.}

\begin{abstract}
  Duality properties are studied for a Gorenstein algebra that is finite and projective over its center.  Using the homotopy category of injective modules, it is proved that there is a local duality theorem for the subcategory of acyclic complexes of such an algebra, akin to the local duality theorems of Grothendieck and Serre in the context of commutative algebra and algebraic geometry. A key ingredient is the Nakayama functor on the bounded derived category of a Gorenstein algebra, and its extension to the full homotopy category of injective modules.\medskip

\end{abstract}


\date{\today}

\thanks{SBI was partly supported by NSF grants DMS-1700985 and DMS-2001368.}

\maketitle

\setcounter{tocdepth}{1}
\tableofcontents

\section{Introduction}
This work is a contribution to the representation theory of Gorenstein algebras, both commutative and non-commutative, with a focus on duality phenomena. The notion of a Gorenstein variety was introduced by Grothendieck~\cite{Grothendieck:1957, Grothendieck:1958, Hartshorne:1966, Hartshorne:1967}, and grew out of his reinterpretation and extension of Serre duality~\cite{Serre:1956} for projective varieties. A local version of his duality is that over a Cohen-Macaulay  local algebra $R$ of dimension $d$, with maximal ideal $\fm$, and for complexes $F,G$ with $F$ perfect, there are natural isomorphisms
\[
\Hom_R(\Ext_R^i(F, G), I(\fm)) \cong \Ext_R^{d-i}(G, R\gam_{\fm}(\omega_R \lotimes_R F))
\]
where  $\omega_R$ is a dualising module, and $I(\fm)$  is the injective envelope of $R/\fm$. The functor $R\gam_{\fm}$ represents local cohomology at $\fm$. Serre duality concerns the case where $R$ is the local ring at the vertex of the affine cone of a projective variety.  The ring $R$ (equivalently, the variety it represents) is  said to be Gorenstein if, in addition, the $R$-module $\omega_R$ is projective. Serre observed that this property is characterised by $R$ having finite self-injective dimension. This result appears in the work of Bass~\cite{Bass:1963} who  gave numerous other characterisations of  Gorenstein rings.  

Iwanaga~\cite{Iwanaga:1978} launched the study of noetherian rings, not necessarily commutative, having finite self-injective dimension on both sides. Now known as Iwanaga-Gorenstein rings, these form an integral part of the representation theory of algebras.  In that domain, the principal objects of interest are maximal Cohen-Macaulay modules and the associated stable category. Auslander~\cite{Auslander:1978a} and Buchweitz~\cite{Buchweitz:1986} have proved duality theorems for the stable category of a Gorenstein algebra with \emph{isolated} singularities.  The driving force behind our work was to understand what duality phenomena can be observed for general Gorenstein algebras.  \intref{Theorem}{ithm:Gor} below is what we found, following Grothendieck's footsteps.

We set the stage to present that result and begin with a crucial definition.

\begin{definition}
  Let $R$ be a commutative noetherian ring. An $R$-algebra $A$ is
  called \emph{Gorenstein} if
  \begin{enumerate}
  \item    the $R$-module $A$ is finitely generated and projective, and
    \item for each $\fp$ in $\Spec R$ with $A_\fp\neq 0$ 
  the ring $A_\fp$ has finite injective dimension as a module over
  itself, on the left and on the right.
\end{enumerate}  
\end{definition}

A Gorenstein $R$-algebra $A$ itself need not be Iwanaga-Gorenstein.
Indeed, for $A$ commutative and Gorenstein, the injective dimension of
$A$ is finite precisely when its Krull dimension is finite, and there
exist rings locally of finite injective dimension but of infinite
Krull dimension. There are precedents to the study of Gorenstein
algebras, starting with~\cite{Bass:1963} and more
recently in the work of Goto and Nishida~\cite{Goto/Nishida:2002}. Our
work differs from theirs in its focus on duality. We refer to
\cite{Gnedin/Iyengar/Krause:2022} for a discussion of examples and natural
constructions  preserving the Gorenstein property.

Let $A$ be a Gorenstein $R$-algebra and
$\omega_{A/R}\colonequals \Hom_R(A,R)$ the dualising
\emph{bimodule}. Unlike in the commutative case, $\omega_{A/R}$ need
not be projective (either on the left or on the right), and the
bimodule structure can be complicated. Nevertheless, it is a tilting
object in $\dcat A$, the derived category of $A$-modules, inducing a
triangle equivalence
\[
\RHom_A(\omega_{A/R},-)\colon \dcat A\longiso\dcat A\,;
\]
see \intref{Section}{se:Gorenstein-dc}. The representation theory of a
Gorenstein algebra $A$ is governed by its maximal Cohen-Macaulay
modules, namely, finitely generated $A$-modules $M$ with
$\Ext_A^i(M,A)=0$ for $i\ge 1$. For our purposes their infinitely
generated counterparts are also important. Thus we consider Gorenstein
projective $A$-modules (abbreviated to G-projective), which are by
definition $A$-modules occuring as syzygies in acyclic complexes of
projective $A$-modules \cite{Buchweitz:1986,Enochs/Jenda:1995}. The
G-projective modules form a Frobenius exact category, and so the
corresponding stable category is triangulated. Its inclusion into the
usual stable module category has a right adjoint, the Gorenstein
approximation functor, $\gp(-)$. The functor
\[
S \colonequals \gp(\omega_{A/R}\otimes_A-)\colon \uGProj A \lra \uGProj A
\]
is an equivalence of triangulated categories, and plays the role of a Serre functor on the subcategory of finitely generated G-projectives. This is spelled out in the result below. Here the $\tExt^i_A(-,-)$ are the Tate cohomology modules, which compute morphisms in $\uGProj A$.

\begin{theorem}
\label{ithm:Gor}
Let $A$ be a Gorenstein $R$-algebra, and let $M,N$ be G-projective $A$-modules with $M$ finitely generated.  For each $\fp\in\Spec R$ there is a natural isomorphism
\[
\Hom_R(\tExt_A^i(M,N),I(\fp))\cong \tExt_A^{d(\fp)-i}(N, \gam_\fp S(M)), 
\]
where $d(\fp)= \dim (R_\fp) - 1$.
\end{theorem}

This is the duality theorem we seek; it is proved in
\intref{Section}{se:GD}. It is new even for commutative rings. The
parallel to Grothendieck's duality theorem is clear.

In the following
we explain the strategy for proving this theorem and some essential
ingredients. The functor $\gam_\fp$ is analogous to the local
cohomology functor encountered above. It is constructed in
\intref{Section}{se:bik} following the recipe in
\cite{Benson/Iyengar/Krause:2008a}, using the natural $R$-action on
$\uGProj A$. Even if $N$ is finitely generated, $\gam_\fp(N)$ need not
be, which is one reason we have to work with infinitely generated
modules in the first place. If $R$ is local with maximal ideal $\fp$
and $A$ has isolated singularities, $\gam_{\fp}$ is the identity and
the duality statement above is precisely the one discovered by
Auslander and Buchweitz.

For a Gorenstein algebra, the stable category of G-projective modules is equivalent to $\KacInj A$, the homotopy category of acyclic complexes of injective $A$-modules. This connection is explained in  \intref{Section}{se:GP},  and builds on the results from \cite{Jorgensen:2005a, Krause:2005}. In fact, much of the work that goes into proving \intref{Theorem}{ithm:Gor} deals with $\KInj A$, the full homotopy category of injective $A$-modules; see \intref{Section}{se:hc}. A key ingredient in all this is the Nakayama functor on the category of $A$-modules:
\[
\kan \colon \Mod A\lra \Mod A\quad\text{where}\quad \kan(M) = \Hom_A(\omega_{A/R},M)
\]
 As noted above, its derived functor induces an equivalence on  $\dcat A$. Following \cite{Krause:2005} we extend the Nakayama functor to all of $\KInj A$, 
which one may think of as a triangulated analogue of the ind-completion of $\dbcat A$.  This \emph{completion} of the Nakayama functor is also an equivalence:
\[
\wh\kan_{A/R}\colon \KInj A\longiso \KInj A\,.
\]
This is proved in \intref{Section}{se:Gorenstein-hc}, where we establish also that it restricts to an equivalence on $\KacInj A$. The induced equivalence on the stable category of G-projective modules is precisely the functor $S$ in the statement of \intref{Theorem}{ithm:Gor}; see \intref{Section}{se:GP} where the singularity category of $A$, in the sense of Buchweitz~\cite{Buchweitz:1986} and Orlov~\cite{Orlov:2004a} also appears. To make this identification, we need to extend results of Auslander and Buchweitz concerning G-approximations; this is dealt with in \intref{Appendix}{se:appendix}. 

Our debt to Grothendieck is evident.  It ought to be clear by now that the work of Auslander and Buchweitz also provides much inspiration for this paper. Whatever new insight we bring is through the systematic use of the homotopy category of injective modules and methods from abstract homotopy theory, especially the Brown representability theorem. To that end we need the structure theory of injectives over finite $R$-algebras from Gabriel thesis~\cite{Gabriel:1962}. Gabriel also introduced the Nakayama functor in representation theory of Artin algebra in his exposition of Auslander-Reiten duality; it is the categorical analogue of the Nakayama automorphism that permutes the isomorphism classes of simple modules over a self-injective algebra \cite{Gabriel:1980}. And it was Gabriel who pointed out the parallel between derived equivalences induced by tilting modules and the duality of Grothendieck and Roos \cite{Keller-Vossieck:1988}.

\subsubsection*{Acknowledgements}

It is a pleasure to thank Xiao-Wu Chen, Vincent G\'elinas, and Janina Letz for various
helpful comments on this work. In addition we are grateful to an
anonymous referee for suggestions concerning the exposition.

\section{Homotopy category of injectives}
\label{se:hc}
In this section we describe certain functors on homotopy categories
attached to noetherian rings.  Our basic references for this material
are \cite{Iyengar/Krause:2006, Krause:2005}.

Throughout $A$ will be ring that is noetherian on both sides; that is
to say, $A$ is noetherian as a left  and as a right $A$-module.  In
what follows $A$-modules will mean left $A$-modules, and
$\op A$-modules are  identified with right $A$-modules. We
write $\Mod A$ for the (abelian) category of $A$-modules and $\mod A$ for its full subcategory consisting of finitely generated modules. Also, $\Inj A$
and $\Prj A$ are the full subcategories of $\Mod A$ consisting of injective and
projective modules, respectively.

For any additive category $\cat A\subseteq\Mod A$, like the ones in the last paragraph, $\bfK(\cat A)$ will denote the associated homotopy category, with its natural structure as a triangulated category. Morphisms in this category are denoted $\KHom[\cat A](-,-)$.  An object $X$ in $\bfK(\cat A)$ is \emph{acyclic} if $\hh X=0$, and the full subcategory of acyclic objects in $\bfK(\cat A)$ is denoted $\bfK_{\mathrm{ac}}(\cat A)$. A complex $X\in\bfK(\cat A)$ is said to be \emph{bounded above} if $X^i=0$ for $i\gg 0$, and \emph{bounded below} if $X^i=0$ for $i\ll 0$.
 
In the sequel our focus in mostly on $\KInj A$, the homotopy category
of injective modules, and its various subcategories; the analogous
categories of projectives play a more subsidiary role.  From work in
\cite{Jorgensen:2005a,Krause:2005,Neeman:2008} we know that the
triangulated categories $\KInj A$ and $\KProj A$ are compactly
generated since the ring $A$ is noetherian on both sides; the compact
objects in these categories are described further below.  Let
$\dcat A$ denote the (full) derived category of $A$-modules and
$\bfq\colon \KMod A\to \dcat A$ the localisation functor; its kernel
is $\bfK_{\mathrm{ac}}(\Mod A)$.  We write $\bfq$ also for its
restriction to the homotopy categories of injectives and
projectives. These functors have adjoints:
\[
\begin{aligned}
\begin{tikzcd}
\KInj A   \arrow[twoheadrightarrow,yshift=.75ex]{r}{\bfq}
  &   \arrow[tail,yshift=-.75ex]{l}{\ires}
 \dcat A
\end{tikzcd}
\end{aligned}
\qquad\text{and}\qquad
\begin{aligned}
\begin{tikzcd}
\KProj A    
  \arrow[twoheadrightarrow,yshift=-.75ex]{r}[swap]{\bfq}
  &   \arrow[tail,yshift=.75ex]{l}[swap]{\pres} \dcat A\,.
\end{tikzcd}
\end{aligned}
\]
Our convention is to write the left adjoint above the corresponding
right one. In what follows it is convenient to conflate $\ires$ and
$\pres$ with $\ires\circ \bfq$ and $\pres\circ\bfq$,
respectively. The images of $\ires$ and $\pres$ are the K-injectives
and K-projectives, respectively. Recall that an object $X$ in
$\KInj A$ is \emph{K-injective} if $\KHom(W,X)=0$ for any acyclic
complex $W$ in
$\bfK(\Mod A)$. We write $\kinj A$ for the full subcategory of $\KInj A$
consisting of K-injective complexes. The subcategory
$\kproj A\subseteq \KProj A$ of K-projective complexes is defined
similarly.

\subsection*{Compact objects}
Since $A$ is noetherian $\Inj A$ is closed under arbitrary direct
sums, and hence so is the subcategory $\KInj A$ of $\KMod A$. As in
any triangulated category with arbitrary direct sums, an object $X$ in
$\KInj A$ is \emph{compact} if $\KHom(X,-)$ commutes with direct
sums. The compact objects in $\KInj A$ form a thick subcategory,
denoted $\KInjc A$. The adjoint pair $(\bfq,\ires)$ above restricts to
an equivalence of triangulated categories
\[
\begin{tikzcd}
\KInjc A   \arrow[rightarrow,yshift=1ex]{r}{\bfq}
 &  \dbcat A\,, \arrow[rightarrow,yshift=-1ex]{l}{\ires}[swap]{\sim}
\end{tikzcd}
\]
where $\dbcat A$ denotes the bounded derived category of $\mod A$; see
\cite[Proposition~2.3]{Krause:2005} for a proof of this assertion. The
corresponding identification of the compact objects in
$\KProj A$ is a bit more involved, and is due to
J{\o}rgensen~\cite[Theorem~3.2]{Jorgensen:2005a}. The assignment
$M\mapsto \Hom_{\op A}(\pres M,A)$ induces an equivalence
\[
\op{\dbcat{\op A}} \longiso \KProjc A.
\]
See also \cite{Iyengar/Krause:2006}, where these two equivalences are related. The formula below for computing morphisms from compacts in $\KInj A$ is useful in the sequel.

\begin{lemma}
\label{lem:an-iso}
For $C,X\in \KInj A$ with $C$ compact, there is a natural isomorphism
\[
\KHom (C,X) \cong H^{0}(\Hom_{A}(\pres C,A) \otimes_{A}X)\,.
\]
\end{lemma}

\begin{proof}
Since $C$ is compact its K-projective resolution $\pres C$ is homotopy equivalent to a complex that is bounded above and consists of finitely generated projective $A$-modules. For each integer $n$, let $X(n)$ be the subcomplex $X^{\ges -n}$ of $X$. Since $X(n)$ is K-injective, the quasi-isomorphism $\pres C\to C$ induces the one on the left
\[
\Hom_{A}(C,X(n)) \longiso  \Hom_{A}(\pres C,X(n)) \longliso \Hom_{A}(\pres C,A)\otimes_{A}X(n)
\]
The one on the right is the standard one, and holds because of the aforementioned properties of $\pres C$ and the fact that $X(n)$ is bounded below. One thus gets a canonical isomorphism
\[
\KHom(C,X(n)) \longiso H^{0}(\Hom_{A}(\pres C,A)\otimes_{A}X(n))\,.
\]
It is compatible with the inclusions $X(n)\subseteq X(n+1)$, so
induces the isomorphism in the bottom row of the following diagram.
\[
\begin{tikzcd}
\KHom(C,\hocolim_{n\ges 0}X(n)) \arrow{d}{\wr} \arrow{r}{\sim} 
	&H^{0}(\Hom_{A}(\pres C,A)\otimes_{A} \hocolim_{n\ges 0}X(n))   \arrow{d}{\wr} \\
\colim_{n\ges 0} \KHom(C,X(n)) \arrow{r}{\sim} 
	& \colim_{n\ges 0} H^{0}(\Hom_{A}(\pres C,A)\otimes_{A}X(n))\,.
\end{tikzcd}
\]
The isomorphism on the left holds by the compactness of $C$, while the one on the right holds because $H^{0}(-)$ commutes with homotopy colimits. It remains to note that $\hocolim_{n\ges 0}X(n) = X$ in $\KInj A$. 
\end{proof}

\subsection*{A recollement}
The functors $\KacInj A\xra{\mathrm{incl}} \KInj A\xra{\bfq}\dcat A$
induce a recollement of triangulated categories
\begin{equation}
\label{eq:recollement}
\begin{tikzcd}
\KacInj A   \arrow[hookrightarrow]{rr}[description]{\mathrm{incl}} && \KInj A
  \arrow[twoheadrightarrow,yshift=-1.5ex]{ll}{\bfr}
  \arrow[twoheadrightarrow,yshift=1.5ex]{ll}[swap]{\bfs}
  \arrow[twoheadrightarrow]{rr}[description]{\bfq} &&\dcat A
  \arrow[tail,yshift=-1.5ex]{ll}{\ires}
  \arrow[tail,yshift=1.5ex]{ll}[swap]{\ipres}
\end{tikzcd}
\end{equation}
The functor $\ires$ is the one discussed above; it embeds $\dcat A$ as the homotopy category of K-injective complexes. The functor $\bfr$ thus has a simple description: there is an exact triangle
\begin{equation}
\label{eq:rho}
\bfr X \lra X \lra \ires X\lra 
\end{equation}
where the morphism $X\to \ires X$ is the canonical one. Indeed,
$\bfr X$ is evidently acyclic and if $W$ is in $\KacInj A$ the
induced map $\KHom (W,\bfr X)\to \KHom(W,X)$ is an isomorphism, for
one has $\KHom (W,\ires X)=0$.

The functor $\ipres\colon \dcat A\to \KInj A$ is fully faithful. The image of $\ipres$ equals the kernel of $\bfs$ and identifies with $\Loc(\ires A)$, the localising subcategory of $\KInj A$ generated by the injective resolution of $A$; see \cite[Theorem~4.2]{Krause:2005}. One may think of $\ipres$ as the
injective version of taking projective resolutions; see \intref{Lemma}{lem:lambda2}. To justify this claim takes  preparation. 

\begin{lemma}
\label{lem:rs}
Restricted to the subcategory $\Loc(\ires A)$ of $\KInj A$ there is a natural isomorphism of functors $\bfr \iso \Sigma^{-1} \bfs\ires$.
\end{lemma}

\begin{proof}
Consider anew the exact triangle \eqref{eq:rho}, but for $X$ in $\Loc(\ires A)$:
\[
\bfr X \lra X \lra \ires X\lra \Sigma\bfr X\,.
\]
Apply $\bfs$ and remember that its kernel is $\Loc(\ires A)$.
\end{proof}

\subsection*{Projective algebras}
In the remainder of this section we assume that the ring $A$ (which
hitherto has been noetherian on both sides) is also projective, as a
module, over some central subring $R$. For the moment the only role
$R$ plays is to allow for constructions of bimodule resolutions with
good properties. Set $\env A\colonequals A\otimes_R\op A$, the
enveloping algebra of the $R$-algebra $A$, and set
\[
E\colonequals \ires_{\env A}A\,.
\]
This is an injective resolution of $A$ as a (left) module over $\env A$. Since $E$ is a complex of $A$-bimodules,  for any complex $X$ of $A$-modules, the right action of $A$ on $E$ induces a left $A$-action on $\Hom_A(E,X)$. The structure map $A\to E$ of bimodules induces a morphism of $A$-complexes
\begin{equation}
\label{eq:AE}
\Hom_A(E,X)\lra \Hom_A(A,X)\cong X \quad\text{for  $X\in \KMod A$.}
\end{equation}

The computation below will be used often:

\begin{lemma}
\label{lem:AE}
The morphism in \eqref{eq:AE} is a quasi-isomorphism for  $X\in \KInj A$.
\end{lemma}

\begin{proof}
  By considering the mapping cone of $A\to E$, the desired statement
  reduces to: For any complex $W\in\KMod A$ that is acyclic and
  satisfies $W^i=0$ for $i\ll 0$, one has $\KHom(W,X)=0$.  Without
  loss of generality we can assume $W^{i}=0$ for $i<0$. Then one gets
  the first equality below
\[\KHom (W,X)  = \KHom (W, X^{\ges -1}) =0,\]
and the second one holds because $X^{\ges -1}$ is K-injective.
\end{proof}

Since $A$ is projective as an $R$-module, $\env A$ is projective as
an $A$-module both on the left and on the right. The latter condition
implies, by adjunction, that as a complex of left $A$-modules $E$
consists of injectives. In particular, for any projective $A$-module
$P$ the $A$-complex $E\otimes_AP$ consists of injective modules. Thus
one has an exact functor
\[
E\otimes_A -\colon \KProj A \lra \KInj A\,.
\]
For each $X$ in $\KInj A$ one has isomorphisms
\begin{align*}
\KHom(E\otimes_A \pres  X,X) 
	& \cong \KHom (\pres X,\Hom_{A}(E,X)) \\
	&\cong \KHom(\pres X, X)\,.
\end{align*}
The second isomorphism is a consequence of \intref{Lemma}{lem:AE} and the
K-projectivity of $\pres X$. Thus, corresponding to the morphism
$\pres X\to X$ there is natural morphism
\begin{equation}
\label{eq:pi}
\pi(X) \colon E\otimes_A \pres  X \lra X
\end{equation}
of complexes of $A$-modules. 

\begin{lemma}
\label{lem:pi-quasiiso}
The morphism $\pi(X)$ in \eqref{eq:pi} is a quasi-isomorphism for  each $X$.
\end{lemma}

\begin{proof}
Let $\eta\colon A\to E$ and $\eps\colon \pres X\to X$ denote the structure maps. These fit in the commutative diagram
\[
\begin{tikzcd}[column sep = large]
A\otimes_A\pres X \arrow{d}[swap]{\eta \otimes_A \pres  X } \arrow{r}{\sim} 
	& \pres  X \arrow{d}{\eps } \\
E \otimes_{A} \pres X \arrow{r}{\pi(X)} 
	& X\,.
\end{tikzcd}
\]
The map $\eta\otimes_A\pres X $ is a quasi-isomorphism as $\eta$ is
one and $\pres X$ is K-projective. Thus $\pi(X)$ is a
quasi-isomorphism.
\end{proof}

\subsection*{The stabilisation functor}
The functor $\bfs \colon \KInj A\to \KacInj A$ from \eqref{eq:recollement} admits the following description in terms of its kernel, which uses the natural transformation $\pi\colon E\otimes_A \pres(-)\to \id$ of functors on $\KInj A$ from \eqref{eq:pi}.

\begin{lemma}
\label{lem:lambda2}
Each  object $X$ in $\KInj A$ fits into an exact triangle
\[
E\otimes_A \pres  X \xra{\ \pi(X) \ } X \lra \bfs X \lra
\]
and this yields a natural isomorphism $E\otimes_A\pres X\iso \ipres X$.
\end{lemma}

\begin{proof}
  Since $\pi(X)$ is a quasi-isomorphism, by
  \intref{Lemma}{lem:pi-quasiiso}, the complex $\bfs X$ is acyclic.
  In $\KProj A$, the complex $\pres X$ is in $\Loc(A)$, and hence in
  $\KInj A$ the complex $E\otimes_A \pres X$ is in $\Loc(E)$. It
  remains to observe that if $W\in \KInj A$ is acyclic then
  $\KHom (E,W)=0$ by \intref{Lemma}{lem:AE}.
\end{proof}

\section{The Nakayama functor and its completion}
\label{se:Nakayam-I}
The Nakayama functor is a standard tool in representation theory of Artin algebras. For instance, the functor interchanges projective and injective modules, thereby providing an efficient method to compute the Auslander-Reiten translate of a finitely generated module \cite{Gabriel:1980}. In this section we discuss  the extension of the Nakayama functor from modules to the homotopy category of injectives.

Throughout the rest of this work we say that a ring $A$ is a
\emph{finite $R$-algebra} if
\begin{enumerate}
\item $R$ is a commutative noetherian ring;
\item $A$ is an $R$-algebra, that is to say, there is a map of rings $R\to A$ whose image is in the centre of $A$;
\item $A$ is finitely generated as an $R$-module.
\end{enumerate}
These conditions imply that $A$ is a noetherian ring, finitely generated as a module over its centre, which is thus also noetherian. Hence $A$ is a
finite algebra over its centre.  When $A$ is a finite $R$-algebra, so is the opposite ring $\op A$.

Let $A$ be a finite $R$-algebra. Following
Buchweitz~\cite[\S7.6]{Buchweitz:1986}, which in turn is inspired by the
terminology in commutative algebra, we call the $A$-bimodule
\[
\omega_{A/R}\colonequals \Hom_R(A,R)
\]
the \emph{dualising bimodule} of the $R$-algebra $A$.  It is finitely generated as an $A$-module, on either side.  Extending the terminology from the context of finite dimensional algebras over fields we call
\begin{equation}
\label{eq:Nakayama}
\kan_{A/R}\colonequals \Hom_A(\omega_{A/R},-) \colon \Mod A\lra \Mod A
\end{equation}
the \emph{Nakayama functor} of the $R$-algebra $A$. Sometimes this
name is used for the functor $\omega_{A/R}\otimes_A-$, which is left
adjoint to $\kan_{A/R}$, but in this work the one above plays a more
central role, hence our choice of nomenclature. When the algebra in
question is clear we drop the ``$A/R$" from subscripts, to write
$\omega$ and $\kan$. In our applications $A$ will be projective as
an $R$-module. Then the left adjoint of $\kan_{A/R}$ is a
Nakayama functor relative to the restriction $\Mod A\to \Mod R$ in the
sense of Kvamme \cite{Kvamme:2020}.

The Nakayama functor can be extended to $\dcat A$, yielding the \emph{derived Nakayama functor}
\[
\RHom_A(\omega,-)\colon \dcat A\lra\dcat A\,.
\]
This functor and its left adjoint has been considered by several authors; see \cite{Happel:1988}. 
Here we study the extension to $\KInj A$, following  \cite[\S6]{Krause:2005}.

The Nakayama functor is evidently additive and admits therefore an
extension to $\KInj A$ as follows.  Extend $\kan$ to $\KMod A$, by
applying it term-wise; denote this functor also $\kan$. Brown
representability yields a left adjoint to the inclusion
$\KInj A\hookrightarrow \KMod A$, say $\lambda$. Set
\begin{equation}
\label{eq:bigNak}
\wh\kan_{A/R} \colon \KInj A\lra \KInj A  
\end{equation}
to be the composite of functors
\[
\KInj A \hookrightarrow \KMod A \xra{\ \kan\ } \KMod A \xra{\ \lambda\ } \KInj A\,.
\]
Our notation is motivated by the fact that $\KInj A$ can be viewed as
a completion of $\dbcat A$, as is explained in
\cite[\S2]{Krause:2005}. The next result is another reason for this
choice. Here $\bfK^{+}(\Inj A)$ denotes the full subcategory of
$\KInj A$ consisting of complexes $W$ that are bounded below. Note that
$\bfK^{+}(\Inj A)\iso \mathbf{D}^{+}(\Mod A)$.

\begin{lemma}
\label{lem:cNak}
On the subcategory $\bfK^{+}(\Inj A)$ there is
an isomorphism of functors
\[
\wh\kan_{A/R}\iso \ires\Hom_A(\omega_{A/R},-)\,.
\]
making the following diagram commutative:
\[
\begin{tikzcd}
  \Mod A \arrow[rightarrow]{d}{\kan} \arrow[tail]{r}{\mathrm{incl}} &
  \mathbf{D}^{+}(\Mod A) \arrow[rightarrow]{d}{\RHom_A(\omega,-)}
  \arrow[tail]{r}{\ires} &
  \KInj A \arrow[rightarrow]{d}{\wh\kan} \\
  \Mod A \arrow[tail]{r}{\mathrm{incl}} & \dcat A
  \arrow[tail]{r}{\ires} & \KInj A\,.
\end{tikzcd}
\]
The functor $\wh\kan_{A/R}\colon \KInj A\to \KInj A$ preserves
arbitrary direct sums, and on compact objects $\wh\kan$ identifies with the functor
\[
\RHom_A(\omega_{A/R},-) \colon \dbcat A\lra \dcat A\,.
\]
\end{lemma}

In general, the above square on the right will not be commutative if one
replaces $\mathbf{D}^{+}(\Mod A)$ by $\dcat A$; confer
\intref{Theorem}{thm:Morita-K}. We examine these functors in greater
detail in the next section.

\begin{proof}
Fix $X\in \bfK^+(\Mod A)$. The key observation is the following.

\begin{claim}
$\lambda X \iso \ires X$, the K-injective resolution of $X$.

\smallskip

Indeed since $X$ is bounded below one can assume that so is $\ires X$, and hence also the mapping cone, say $Z$, of the morphism $X\to \ires X$. Since $Z$ is also acyclic, arguing as in the proof of \intref{Lemma}{lem:AE} one gets that $\KHom(Z,Y)=0$ for any $Y\in \KInj A$. Thus the
morphism $X\to \ires X$ induces an isomorphism
\[
\KHom(\ires X,Y) \iso \KHom(X,Y)\,,
\]
and this justifies the claim. 
\end{claim}

When $X$ is bounded below so is $\Hom_A(\omega,X)$. Thus the claim
above yields
\[
\wh\kan(X) = \lambda \Hom_A(\omega, X) \cong \ires \Hom_A(\omega, X)\,.
\]

Now fix  $X\in \mathbf{D}^{+}(\Mod A)$. Again, one can assume $\ires
X$ is also bounded below, and therefore
\[ 
\wh\kan(\ires X)= \lambda \Hom_A(\omega,\ires X) \cong
  \ires\Hom_A(\omega,\ires X)=\ires\RHom_A(\omega,X) \,.
\]
This yields the commutativity of the right hand square.

For the second part of the lemma, it remains to note that the
functor $\kan$ preserves direct sums, as the $A$-module $\omega$ is
finitely generated, and $\lambda$ preserves direct sums, as it is a
left adjoint.
\end{proof}

\section{Gorenstein algebras and their derived categories}
\label{se:Gorenstein-dc}

In this section we introduce Gorenstein algebras and characterise them
in terms of the derived Nakayama functor. This generalises a
well-known fact for Artin algebras.  In that case the algebra is
Gorenstein if and only if the dualising module is a tilting module so
that the derived Nakayama functor is an equivalence.

\subsection*{Commutative Gorenstein rings}
A commutative noetherian ring $R$ is \emph{Gorenstein} if for each prime (equivalently, maximal) ideal $\fp$, the local ring $R_\fp$ has finite injective dimension as a module over itself \cite{Bass:1963}. When the Krull dimension of $R$ is finite, this condition is equivalent to $R$ itself having finite injective dimension; see \cite[Theorem, \S1]{Bass:1963} for details.
 
\subsection*{Gorenstein algebras}
We say that a ring $A$ is a \emph{Gorenstein} $R$-algebra if
\begin{enumerate}[\quad\rm(1)]
\item
$A$ is a finite $R$-algebra;
\item
$A$ is projective as an $R$-module;
\item
$A_\fp$ is Iwanaga-Gorenstein for each  $\fp\in\Spec R$ with
$A_\fp\neq 0$.
\end{enumerate}
Condition (3) means $A_\fp$ has finite injective dimension as a module
over itself, on the left and on the right; then the injective
dimensions coincide; see \cite[Lemma~A]{Za1969}.

The following lemma provides a comparison between $A$ and $R$ 
with respect to the Gorenstein property.

\begin{lemma}
\label{lem:RGor}
Let $A$ be a Gorenstein $R$-algebra and  $\fp\in\Spec R$.
Then the  ring $R_\fp$ is Gorenstein whenever $A_\fp\neq 0$.
\end{lemma}

\begin{proof}
  As the $R$-module $A$ is projective so is the $R_{\fp}$-module
  $A_{\fp}$, and hence for each finitely generated $R_\fp$-module $M$
  one has the isomorphism below
\[
  \Ext_{R_\fp}^i(M,R_\fp)\otimes_{R_{\fp}} A_{\fp} \cong
  \Ext_{A_\fp}^i(M \otimes_{R_\fp} A_\fp,A_\fp)=0 \qquad\text{for
    $i\gg 0$}.
\]
The equality on the right holds because the injective dimension of
$A_\fp$ is finite. We deduce from the computation above that
$\Ext_{R_\fp}^i(M,R_\fp)=0$ for $i \gg 0$, since $A_\fp\neq 0$. Hence
$R_\fp$ is Gorenstein; see \cite[Proposition
3.1.14]{Bruns/Herzog:1998a}.
\end{proof}

Let $A$ be a finite $R$-algebra that is projective as an
$R$-module. Then $R$ admits a decomposition
$R'\times R^{\prime\prime}$ such that $A_\fp\neq 0$ for all $\fp\in\Spec R'$ and $A$ is
finitely generated over $R'$. Thus one may assume that $A$ is faithful
as an $R$-module, and then the Gorenstein property for $A$ implies
that $R$ is Gorenstein; see \cite{Gnedin/Iyengar/Krause:2022} for details.

The preceding result has also a converse, but this plays no role in
the sequel so we discuss this at the end of this section; see
\intref{Theorem}{thm:Gconverse}. The Gorenstein condition is reflected
also in the dualising bimodule of the $R$-algebra $A$. To discuss
this, we recall some aspects of perfect complexes over finite
algebras.

Let $A$ be a finite $R$-algebra and $M$ a complex of $A$-modules. Recall that $M$ is \emph{perfect} if it is isomorphic in $\dcat A$ to a bounded complex of finitely generated projective $A$-modules; equivalently, $M$ is compact, as an object in the triangulated category $\dcat A$; equivalently, $M$ is in $\Thick(A)$; see \cite[Theorem~2.2]{Neeman:1992}.

The following criterion for detecting perfect complexes will be handy.

\begin{lemma}
\label{lem:perfect}
Let $A$ be a finite $R$-algebra. For $M\in\dbcat A$ the following conditions are
equivalent.
\begin{enumerate}[\quad\rm(1)]
\item $M$ is perfect in $\dcat{A}$
\item $M_\fm$ is perfect in $\dcat{A_\fm}$ for each maximal ideal $\fm$ in $R$.
\item $\Tor^A_i(L,M)=0$  for each $L\in\mod \op A$ and $i\gg 0$.
\end{enumerate}
\end{lemma}

\begin{proof}
The equivalence of (1) and (2) is  due to Bass~\cite[Proposition III.6.6]{Bass:1968}.  Evidently (1) implies (3), and the reverse implication can be verified by an argument akin to that for \cite[Theorem A.1.2]{Avramov/Iyengar:2019a}.
\end{proof}

\begin{remark}
We say that a complex $M$ of $A$-bimodules is \emph{perfect on both sides} if it is perfect both in $\dcat A$ and in $\dcat{\op A}$; said otherwise,  the restriction of $M$ along either map $A\to \env A \gets \op A$ is perfect, in the corresponding category.

We note also that when $M$ is a complex of $A$-bimodules,
$\RHom_A(M,A)$ has a left $A$-action induced by the right $A$-action
on $M$, and a right action induced by the right $A$-action on $A$. In
our context $A$ is a projective $R$-module, so one can realise
$\RHom_A(M,A)$ as a complex of bimodules, namely, the complex
$\Hom_A(M,\ires_{\env A}A)$.
\end{remark}

\begin{lemma}
\label{lem:biperfect}
Let $A$ be a finite $R$-algebra and $M$ a complex of $A$-bimodules that is perfect on both sides. The following statements hold:
\begin{enumerate}[\quad\rm(1)]
 \item
 There exists a quasi-isomorphism $P\to M$ of $A$-bimodules where $P$ is bounded, consisting of finitely generated $A$-bimodules that are projective on both sides.
\item
When $A$ is a Gorenstein $R$-algebra, $\RHom_A(M,A)$ is perfect on both sides.
 \end{enumerate}
\end{lemma}

\begin{proof}
(1) The hypothesis on $M$ implies that the $\env A$-module $H^*(M)$ is  finitely generated. There thus exists a projective $\env A$-resolution, say $Q\to M$ with each $Q_i$ finitely generated and $0$ for $i\ll 0$. Fix an integer
\[
i\ge \max\{\pdim_{A}M,\pdim_{\op A}M\}\,.
\]
The morphism $Q\to M$ factors through the quotient complex
\[
P\colonequals 0\lra \Coker(d^Q_{i+1})\lra Q_i \lra Q_{i-1}\lra \cdots 
\]
Since $A$-modules $Q_i$ are projective on both sides, it follows by the choice of $i$ that so is the $A$-module $\Coker(d^Q_{i+1})$. Thus $P$ is the complex we seek.

(2) That $\RHom_A(M,A)$ is perfect on the right is clear; for example, it is equivalent to $\Hom_A(P,A)$ with $P$ as above; this does not involve the Gorenstein property.

As to the perfection on the left, by \intref{Lemma}{lem:perfect} it suffices to check the perfection locally on $\Spec R$. Thus  we can assume that the injective dimension of $A$ is finite. For any finitely generated $\op A$-module $L$ one has a natural isomorphism
\[
L\lotimes_A \RHom_A(M,A) \longiso \RHom_A(\RHom_{\op A}(L,M),A)\,.
\]
Since $M$ is perfect over $\op A$ and $A$ has finite injective dimension (on the right),  so does $M$ and hence $\hh {\RHom_{\op A}(L,M)}$ is bounded. Then the finiteness of the injective dimension of $A$ on the left implies that 
\[
\hh{\RHom_A(\RHom_{\op A}(L,M),A)}
\]
 is bounded. It thus follows from the quasi-isomorphism above that
 \[
\Tor_i^{A}(L,\RHom_A(M,A))=0\qquad\text{for $|i|\gg 0$.}
\]
This implies  $\RHom_A(M,A)$ is perfect on the left; see \intref{Lemma}{lem:perfect}.
\end{proof}

\subsection*{An equivalence of categories}
Let $A$ be a Gorenstein $R$-algebra, $\omega_{A/R}$ its dualising module, and $\kan_{A/R}$ the Nakayama functor; see~\eqref{eq:Nakayama}.  As for finite dimensional algebras~\cite{Happel:1991a} the derived functor of the Nakayama functor is an auto-equivalence of the bounded derived category. In other words, $\omega_{A/R}$ is a tilting complex for $A$.

\begin{theorem}
\label{thm:Morita-D}
Let $A$ be a Gorenstein $R$-algebra. The $A$-bimodule $\omega_{A/R}$ is perfect on both sides, and induces adjoint  equivalences of triangulated categories
\[
\begin{tikzcd}[column sep=large]
\dcat A   \arrow[rightarrow,yshift=1ex]{rr}{\omega_{A/R}\lotimes_A-}   && \dcat A\,.
  \arrow[rightarrow,yshift=-1ex]{ll}{\RHom_A(\omega_{A/R},-)}[swap]{\sim}
\end{tikzcd}
\]
Moreover, these restrict to adjoint equivalences on $\dbcat A$.
\end{theorem}

\begin{proof}
The argument becomes a bit more transparent once we consider the ring $E\colonequals \End_A(\omega)$, and its natural left action on $\omega$ that is compatible with the left $A$-module structure. We verify first the following properties of $\omega$:
\begin{enumerate}
\item
The natural maps $A\to \op E$ and $A\to \End_{E}(\omega)$ of rings are isomorphisms.
\item
$\Ext_A^i(\omega,\omega)=0=\Ext_E^i(\omega,\omega)$ for $i\ge 1$.
\item
$\omega$ is compact both in $\dcat A$ and in $\dcat E$.
\end{enumerate}
The first map in (1) is
\[
A\lra \op{\End_A(\omega)} \quad \text{where $a\mapsto (w\mapsto wa)$}
\]
A routine computation reveals that this is indeed a map of rings. Its bijectivity follows from the computation:
\begin{align*}
\RHom_A(\omega,\omega) 
	&\cong \RHom_R(\Hom_R(A,R),R) \\
	& \cong \Hom_R(\Hom_R(A,R),R) \\
	& \cong A
\end{align*}
where the first isomorphism is adjunction, and the others hold because the $R$-module $A$ is finite and projective. The computation above also establishes that $\Ext^i_A(\omega,\omega)=0$ for $i\ge 1$. The  justifices the first parts of the (1) and (2). Given that $A\iso \op E$, applying the already established part of the result to $\op A$ completes the argument for (1) and (2).

It remains to verify (3), and again, given that $E\cong \op A$ as rings, it suffices to check that $\omega$ is perfect in $\dcat A$. Since the $A$-module $\omega$ is finitely generated it suffices to prove that it has finite projective dimension as an $A$-module.  By \intref{Lemma}{lem:perfect} it suffices to verify that the $A_\fp$-module $M_\fp$ has finite projective dimension for each $\fp\in \Spec R$.  Since
\[
\Hom_{R_\fp}(A_\fp,R_\fp) \cong \Hom_{R}(A,R)_\fp
\]
as $A_\fp$-bimodules, and $A_\fp$ is a Gorenstein $R_\fp$-algebra, replacing $R$ and $A$ by their localisations at $\fp$ we can assume that $(R,\fm,k)$ is a local ring and $A$ is a Gorenstein $R$-algebra of finite injective dimension; the desired conclusion is that the projective dimension of $\Hom_R(A,R)$ is finite. At this point one can invoke \cite[Proposition~7.6.3(ii)]{Buchweitz:1986} to complete the proof. The proof of \emph{op.~cit.} uses the theory of Cohen-Macaulay approximations. Here is a direct argument:

Since $R$ is Gorenstein, by \intref{Lemma}{lem:RGor}, and local, it has
finite injective dimension; choose a finite injective resolution
$R\to\ires R$. Choose also a finite injective resolution
$A\to\ires A$. Then $\Hom_R(\ires A,\ires R)$ is a bounded
complex of flat $A$-modules, quasi-isomorphic to $\Hom_R(A,R)$; thus
the $A$-module $\Hom_R(A,R)$ has finite flat dimension. Since it is
also finitely generated, it follows that its projective dimension is
finite; see \intref{Lemma}{lem:perfect}.

This completes the proofs of assertions (1)--(3). 

Next we verify the stated equivalence of (the full derived) categories. This is a standard argument, given the properties of $\omega$. Here is a sketch: To begin with, given the isomorphism  $A\cong \op E$ of rings, the stated adjunction can be factored as
\[
\begin{tikzcd}[column sep=large]
 \dcat A \arrow[rightarrow,yshift=.75ex]{rr}{ \RHom_A(\omega,-)}
 	&& \dcat {\op E} \arrow[rightarrow,yshift=-.75ex]{ll}{-\lotimes_E\omega }
	& \arrow[leftarrow]{l}[swap]{\sim} \dcat A\,.
\end{tikzcd}
\]
It thus suffices to verify that the adjoint pair on the left are quasi-inverses to each other, that is to say that their counit and unit of the adjunction are isomorphisms. The counit is  the evaluation map
\[
\varepsilon(M)\colon \RHom_A(\omega, M)\lotimes_E \omega \lra M \quad\text{for $M$ in $\dcat A$.}
\]
The map above is an isomorphism for it factors as the composition of isomorphisms
\begin{align*}
\RHom_A(\omega, M)\lotimes_E \omega 
	& \longiso  \RHom_A(\RHom_E(\omega,\omega),M) \\
	& \longiso \RHom_A(A,M) \\
	& \longiso M
\end{align*}
where the first map is standard, and is an quasi-isomorphism because $\omega$ is compact in $\dcat{E}$, by (3) above, and the second map is induced by the natural map $A\to \RHom_E(\omega,\omega)$ that is a quasi-isomorphism because of properties (1) and (2). Similarly, the unit map
\[
N \lra \RHom_A(\omega, N\lotimes_E \omega) 
\]
is a quasi-isomorphism for all $N$ in $\dcat A$ for it factors as the composition 
\[
N\longiso N\otimes_E \RHom_A(\omega,\omega) \longiso \RHom_A(\omega, N\otimes_E\omega)
\]
where the first map is induced by the isomorphism $E\iso \RHom_A(\omega,\omega)$, and the second one is standard, and is an isomorphism because $\omega$ is perfect in $\dcat A$.

This completes the proof that the stated adjoint pair of functors induce an equivalence on $\dcat A$. It remains to note that for each $M$ in $\dbcat A$ the $A$-complex $\RHom_A(\omega,M)$ and $\omega\lotimes_AM$ are in $\dbcat A$ as well, because $\omega$ is compact on both sides. Thus they restrict to adjoint equivalences on $\dbcat A$.
\end{proof}

We can now offer converses to \intref{Lemma}{lem:RGor}; see Goto~\cite{Goto:1982} for a similar statement in commutative algebra. Regarding condition (3), it is noteworthy that the injective dimension of $A$ need not be finite; so there need not be a global bound (independent of $M$) on the degree $i$ beyond which $\Ext^i_A(M,A)$ is zero. Indeed, there exist even commutative Gorenstein rings $R$ that exhibit this phenomenon; see \cite[A1]{Nagata:1975}.

\begin{theorem}
\label{thm:Gconverse}
Let $R$ be a commutative noetherian Gorenstein ring, and $A$ a finite, projective, $R$-algebra. The following conditions are equivalent.
\begin{enumerate}[\quad\rm(1)]
\item
The $R$-algebra $A$ is Gorenstein.
\item
The $A$-bimodule $\omega_{A/R}$ is perfect on both sides.
\item
  For each $M\in\mod A$ and
 $N\in\mod \op A$, we have
$\Ext^i_A(M,A)=0$ for $i\gg 0$ and  $\Ext^i_{\op A}(N,A)=0$  for $i\gg 0$.
\item
The functors $\RHom_A(-,A)$ and $\RHom_{\op A}(-,A)$ induce triangle equivalences 
\[
\begin{tikzcd}[column sep=large]
\op{\dbcat A} \arrow[rightarrow,yshift=1ex]{rr}{ \RHom_A(-, A)}
 	&& {\dbcat {\op A}}
        \arrow[rightarrow,yshift=-1ex]{ll}{\RHom_{\op A}(-,A) }[swap]{\sim}\,.
\end{tikzcd}
\]
\end{enumerate}
\end{theorem}

\begin{proof}
The proof that (1)$\Rightarrow$(2)  is contained in \intref{Lemma}{lem:RGor} and \intref{Theorem}{thm:Morita-D}. 

(2)$\Rightarrow$(1) The hypotheses are local with respect to primes in $\Spec R$, as is the conclusion, by definition. We may thus assume $R$ is local and hence of finite injective dimension. Then, since $A$ is a projective $R$-module, it follows from  adjunction that the $A$-module $\omega=\Hom_R(A,R)$ has finite injective dimension on both sides. For the same reason, one gets that the following natural map is a quasi-isomorphism
\[
 A \lra \RHom_{\op A}(\omega,\omega)\,;
\]
see the proof of \intref{Theorem}{thm:Morita-D}. As $\omega$ is perfect on the right, it is in $\Thick(A)$ in $\dcat{\op A}$, and  the quasi-isomorphism above implies that $A$ is in $\Thick(\omega)$ in $\dcat{A}$. In particular, since the injective dimension of $\omega$ as a left $A$-module is finite so is that of $A$. Similarly, we deduce that the injective dimension of $A$ is finite also on the right.

(1)$\Rightarrow$(3) Suppose $A$ is a Gorenstein $R$-algebra and fix an
$M$ in $\mod A$. Since $A$ is in $\Thick(\omega)$ in $\dbcat A$, it
suffices to verify that $\Ext^i_A(M,\omega)$ for $i\gg 0$.  Adjunction
yields
\[
\Ext^i_A(M,\omega) = \Ext^i_A(M,\Hom_R(A,R)) \cong \Ext^i_R(M,R)\,.
\]
As $R$ is Gorenstein, by \intref{Lemma}{lem:RGor}, the problem reduces
to the commutative case, where the result is due to
Goto~\cite[Theorem~1]{Goto:1982}. The same argument gives the result
for $N$ in $\mod\op A$.

(3)$\Rightarrow$(1) For each prime $\fp$ in $\Spec R$ and $M$ in $\mod
A$ we have an isomorphism
\[\Ext^i_A(M,A)_\fp \cong \Ext^i_{A_\fp}(M_\fp,A_\fp) \qquad (i\ge
  0)\,.\]
If this  vanishes for each $M$ and $i\gg 0$, then $A_\fp$ has finite
injective dimension as a left $A_\fp$-module. Analogously,  $A_\fp$ has finite
injective dimension as a right $A_\fp$-module. Thus $A$ is Gorenstein. 

(1)$\Rightarrow$(4) For each $M\in \dbcat A$, the $\op A$-complex
$\RHom_A(M,A)$ belongs to  $\dbcat{\op A}$, by the already verified implication (1)$\Rightarrow$(3), so it remains to verify that the natural biduality morphism
\[
M \lra \RHom_A(\RHom_A(M,A),A)
\]
is an isomorphism. Since $\RHom_A(M,A)$ is in $\dbcat{\op A}$ this can
be checked locally on $\Spec R$, where it holds for the injective
dimension of $A$ is locally finite. The same argument gives the result
for $N$ in $\dbcat{\op A}$.

(4)$\Rightarrow$(3) Clear.
\end{proof}

\begin{remark}
The argument in the proof of \intref{Theorem}{thm:Gconverse} raises the question: When $A$ is a Gorenstein $R$-algebra, is $\omega_{A/R}$ generated by $A$ in $\dbcat{\env A}$, that is to say, is it in $\Thick_{\env A}(A)$? By standard arguments, this question is equivalent to: Is 
\[
\RHom_R(A\lotimes_{\env A}A,R) \iso \RHom_{\env A}(A,\omega_{A/R}) 
\]
perfect as a dg module over $\mcE\colonequals \RHom_{\env A}(A,A)$, the (derived) Hochschild cohomology algebra? When this conditions holds it would follow from the isomorphism above that if $\mathrm{HH}^i(A/R)=0$ for $i\gg 0$, then also $\mathrm{HH}_i(A/R)=0$ for $i\gg 0$. 

This turns out not to be the  case when $A$ is finite dimensional and self-injective over a field: Let $k$ be a field,  $q\in k$ an element that is nonzero and not a root of unity, and set
\[
\Lambda\colonequals \frac{k\langle x,y\rangle}{(x^2,xy+qyx,y^2)}\,.
\]
 Then Buchweitz, Madsen, Green, and Solberg~\cite{Buchweitz/Green/Madsen/Solberg:2005} prove that $\rank_k \mathrm{HH}^*(A/k)=5$ whereas $\mathrm{HH}_i(A/k)$ is nonzero for each $i\ge 0$.
 
On the other hand the question has, trivially, a positive answer when $A$ is a symmetric $R$-algebra, that is to say, when $\omega_{A/R}\cong A$ as an $A$-bimodule. So this begs the question: If $\omega_{A/R}$ is in $\Thick_{\env A}(A)$, is then $A$ a symmetric $R$-algebra?

\end{remark}

\section{Gorenstein algebras and their homotopy categories}
\label{se:Gorenstein-hc}

Let $A$ be a Gorenstein $R$-algebra. We study in this case the properties
of the Nakayama functor for the homotopy category of injectives $\KInj A$.

\subsection*{The Nakayama functor}
As explained in \intref{Section}{se:Nakayam-I}, the Nakayama functor admits a canonical extension to a functor
$\wh\kan_{A/R} \colon \KInj A\to \KInj A$. The following result
discusses the compatiblity of this functor with the recollement for
$\KInj A$ introduced in \eqref{eq:recollement} and the equivalence on
$\dcat A$ in \intref{Theorem}{thm:Morita-D}.

\begin{theorem}
\label{thm:Morita-K}
Let $A$ be a Gorenstein $R$-algebra. The functor
$\wh\kan_{A/R} \colon \KInj A\to \KInj A$ is a triangle equivalence
making the following square commutative:
\[
\begin{tikzcd}
\mathbf{D}(\Mod A)   \arrow[rightarrow]{d}{\RHom_A(\omega,-)}
\arrow[tail]{r}{\ires} &
\KInj A \arrow[rightarrow]{d}{\wh\kan} \\
\dcat A    \arrow[tail]{r}{\ires} & \KInj A 
\end{tikzcd}
\]
Moreover $\wh\kan_{A/R}$ restricts to
an equivalence $\KacInj A\iso \KacInj A$.
\end{theorem}

The key step in the proof of the result is a ``concrete" description
of $\wh\kan$; see \intref{Lemma}{lem:cNak+} below. To that end note
that \intref{Lemma}{lem:biperfect} applies to the dualising
bimodule $\omega_{A/R}$; fix a complex $P$ provided by that result and
set $\wh\omega_{A/R}\colonequals P$. Thus
\[
\wh\omega_{A/R} \lra \omega_{A/R}
\]
is a finite resolution of $\omega_{A/R}$ by finitely generated $A$-bimodules that are projective on either side. This implies, in particular,  that
when $X$ is a complex of injective $A$-modules, so is $\Hom_A(\wh\omega_{A/R},X)$; this follows from the standard Hom-tensor adjunction, and requires only that $\wh\omega_{A/R}$ consists of modules projective on the right. One thus has the induced exact functor
\[
 \Hom_A(\wh\omega_{A/R},-)\colon \KInj A\to \KInj A\,. 
\]
Here is the vouched for description of the completion of the Nakayama functor.

\begin{lemma}
\label{lem:cNak+}
The quasi-isomorphism $\wh\omega_{A/R}\to\omega_{A/R}$ induces  an isomorphism 
\[
\wh\kan_{A/R} \iso \Hom_A(\wh\omega_{A/R},-)
\]
of functors  on $\KInj A$.
\end{lemma}

\begin{proof}
For  $X\in \KInj A$ the morphism $\wh\omega \to\omega$ induces the morphism 
\[
\Hom_A(\omega,X) \lra  \Hom_A(\wh\omega,X)
\]
of complexes of $A$-modules. Since $\Hom_A(\wh\omega,X)$ consists of injective modules, one gets an induced morphism
\[
\wh\kan(X) = \lambda \Hom_A(\omega,X) \lra \Hom_A(\wh\omega,X)\,.
\]
This is the natural transformation in question. The functors $\wh\kan$ and $\Hom_A(\wh\omega,-)$ preserve arbitrary direct sums, the former by \intref{Lemma}{lem:cNak} and the latter because $\wh\omega$ is a bounded complex of finitely generated modules, by choice. Thus  it suffices to verify that the morphism above is an isomorphism when $X$ is compact in $\KInj A$, that is to say, when it is of the form $\ires M$, for some $M\in\dbcat A$. In this case the morphism in question is the composite
\[
\wh\kan(\ires M) \iso \ires \Hom_A(\omega,\ires M) \to \Hom_A(\wh\omega,\ires M)\,,
\]
where the isomorphism is taken from \intref{Lemma}{lem:cNak}. The map above is a quasi-isomorphism and its source and target are K-injective; the former by construction and the latter because $\wh\omega$ is a bounded complex of projectives. It remains to observe that a quasi-isomorphism between K-injectives is an isomorphism in $\KInj A$.
\end{proof}

\begin{proof}[Proof of \intref{Theorem}{thm:Morita-K}]
Given \intref{Lemma}{lem:cNak+}, a standard d\'evissage argument shows that $\wh\kan$ is a triangle equivalence: the functor preserves arbitrary direct sums and identifies with $\RHom_A(\omega, -)$ when restricted to compacts, by \intref{Lemma}{lem:cNak}. It remains to note that $\RHom_A(\omega,-)$ is an equivalence on $\dbcat A$, by \intref{Theorem}{thm:Morita-D}. 

For the commutativity of the square, fix a complex $X\in\dcat A$. We have already seen in \intref{Lemma}{lem:cNak} that
\[
\wh\kan(\ires X) \iso \ires \RHom_A(\omega, X)
\]
when $X$ is bounded below. An arbitrary complex in $\dcat A$ is quasi-isomorphic to a homotopy limit of bounded below complexes. Thus it remains to observe that both functors preserve  homotopy limits. 

It remains to verify that $\wh\kan$ restricts to an equivalence between acyclic complexes; equivalently that a complex $X\in\KInj A$ is acyclic if and only if $\wh\kan(X)$ is acyclic.  

Since $\wh\omega$ is perfect on the left, $\wh\kan$ preserves acyclic complexes. On the other hand, since $\RHom_A(\omega, A)$ is in $\Thick(A)$ in $\dbcat A$ by
\intref{Lemma}{lem:biperfect}, it follows that $\wh\kan(\ires A)$ is in
$\Thick(\ires A)$. Using the isomorphism
\[
H^n(X)\cong\Hom_{\bfK}(\ires A,\Sigma^n X)\cong \Hom_{\bfK}(\wh\kan(\ires A), \Sigma^n\wh\kan(X))
\]
it follows that when $\wh\kan(X)$ is acyclic so is $X$. 
\end{proof}

\begin{remark}
  One may turn $\dbcat A$ into a dg category such that $\KInj A$
  identifies with its derived category; see
  \cite[Appendix~A]{Krause:2005}. Then $\wh\kan_{A/R}$ identifies with
  the lift of the Nakayama functor $\dbcat A\to\dbcat A$.
\end{remark}

\begin{remark}
\label{rem:Morita-KP}
If $X$ is a complex of projective $A$-modules, then so is the $A$-complex $\wh\omega_{A/R}\otimes_AX$; this is because $\wh\omega$ consists of modules projective on the left. Thus one gets an exact functor
\[
\wh\omega_{A/R}\otimes_A - \colon \KProj A\lra \KProj A\,.
\]
Arguing as in the proof of \intref{Theorem}{thm:Morita-K} one can verify that this is also an equivalence of categories. 
\end{remark}

Since the Nakayama functor $\wh\kan_{A/R}$ is an equivalence, it has a quasi-inverse. This is described below.

\subsection*{A quasi-inverse}
Set $V\colonequals \Hom_A(\wh\omega_{A/R},A)$; this is a bounded complex of $A$-bimodules where the left action is through the right $A$-module structure on $\wh\omega_{A/R}$ and the right action is through the right $A$-module structure of $A$.

\begin{proposition}
\label{pr:quasiinverse}
The assignment $X\mapsto \Hom_A(V,X)$ induces an exact functor 
\[
\Hom_A(V,-)\colon \KInj A\lra \KInj A\,.
\]
This functor is a quasi-inverse of $\wh\kan_{A/R}$, and so an equivalence of categories.
\end{proposition}

\begin{proof}
The complex $\wh\omega$ consists of modules projective on the left and the right $A$-action on  $V=\Hom_A(\wh \omega,A)$ is through $A$, so $V$ consists of modules that are projective on the right. Given this it is easy to verify that $\Hom_A(V,-)$ maps complexes of injectives to complexes of injectives so induces an exact functor on $\KInj A$. For $X\in \KInj A$ the natural morphism of complexes
\[
V\otimes_A X = \Hom_A(\wh \omega, A)\otimes_A X \lra \Hom_A(\wh\omega,X)
\]
is an isomorphism because the complex $\wh\omega$ is a bounded complex of modules projective on the left. This justifies the second  isomorphism below:
\begin{align*}
\KHom(X,\Hom_A(V, \Hom_A(\wh\omega,X))) 
&\cong \KHom(V\otimes_AX, \Hom_A(\wh\omega,X)) \\
&\cong \KHom(\Hom_A(\wh\omega,X),\Hom_A(\wh\omega,X))\,.
\end{align*}
The first one is adjunction. Thus the identity  on $\Hom_A(\wh\omega,X)$ induces a morphism
\[
\eta(X)\colon X\lra \Hom_A(V, \Hom_A(\wh\omega,X))
\]
which is natural in $X$. As functors of $X$, both the source and the target of $\eta$ are exact and preserves direct sums; thus, to verify that $\eta(X)$ is an isomorphism for each $X$ it suffices to verify that this is so for compact objects in $\KInj A$, that is to say, for the induced natural transformation on $\dbcat A$. This is the map
\[
M\mapsto \RHom_A(\RHom_A(\omega,A),\RHom_A(\omega,M))\,.
\]
Since $\omega$ and $\RHom_A(\omega,A)$ are perfect as complexes of left $A$-modules, by \intref{Theorem}{thm:Morita-D} and \intref{Lemma}{lem:biperfect}, respectively, the map above can be obtain by applying $(-)\lotimes_AM$ to the natural homothety morphism 
 \[
A\lra \RHom_A(\RHom_A(\omega,A),\RHom_A(\omega,A))\,.
 \]
Observe this a morphism in $\dbcat{\env A}$. It remains to note that the map above is a quasi-isomorphism by, for example, \intref{Theorem}{thm:Morita-D}.
\end{proof}

\subsection*{Acyclicity versus total acyclicity}

Set $E\colonequals \ires_{\env A}A$, the injective resolution of $A$ as an $A$-bimodule, and consider adjoint functors
\[
\begin{tikzcd}
\KProj A \arrow[hookrightarrow,yshift=.75ex]{rr}
 	&&\KFlat A \arrow[twoheadrightarrow,yshift=-.75ex]{ll}{\bff } \arrow[rightarrow,yshift=.75ex]{rr}{E\otimes_A-}
		 &&\KInj  A \arrow[rightarrow,yshift=-.75ex]{ll}{\Hom_A(E,-) }
\end{tikzcd}
\]
where $\bff$ is the right adjoint to the inclusion. It exists because $\KProj A$ is a compactly generated triangulated category and its inclusion in $\KFlat A$ is compatible with coproducts; see \cite[Proposition~2.4]{Iyengar/Krause:2006}. One thus gets an adjoint pair
\[
\begin{tikzcd}
\KProj A \arrow[rightarrow,yshift=.75ex]{rr}{\bft}
		 &&\KInj  A \arrow[rightarrow,yshift=-.75ex]{ll}{\bfh}
\end{tikzcd}
\]
where  $\bft\colonequals E\otimes_A-$ and $\bfh\colonequals \bff\circ\Hom_A(E,-)$.

Let $\cat A$ be an additive category. A complex $X\in\bfK(\cat A)$ is called \emph{totally acyclic} if $\Hom(W,X)$ and $\Hom(X,W)$ are
acyclic complexes of abelian groups for all $W\in\cat A$. We denote by $\bfK_{\mathrm{tac}}(\cat A)$ the full subcategory of totally acyclic complexes.

\begin{theorem}
\label{thm:documenta}
Let $A$ be a Gorenstein $R$-algebra.  The adjoint functors
$(\bft,\bfh)$ above are equivalences of categories, and they restrict
to equivalences
\[
\begin{tikzcd}
\KacProj A  \arrow[rightarrow,yshift=1ex]{rr}{\bft}
 	&& \KacInj A\,. \arrow[rightarrow,yshift=-1ex]{ll}{\bfh}[swap]{\sim}
\end{tikzcd}
\]
Moreover, there are equalities
\[
 \bfK_{\mathrm{tac}}(\Prj A)=\KacProj A  \quad\text{and}\quad \bfK_{\mathrm{tac}}(\Inj A) = \KacInj A\,.
\]
\end{theorem}

\begin{proof}
  It is clear that the functor $\bft$ preserves direct sums. It also
  preserves compact objects, as we explain now. We may assume that a compact object in
  $\KProj A$ is of the form $\Hom_A(\pres M,A)$ for some
  $M\in\mod\op A$. This yields a complex
\[E\otimes_A \Hom_{\op A}(\pres M,A)\cong \Hom_{\op A}(\pres M, E)\]
which is compact in $\KInj  A$, because it is bounded below 
with \[H^i \Hom_{\op A}(\pres M, E)\cong\Ext_{\op A}^i(M,A)=0\] for $i\gg 0$, by \intref{Theorem}{thm:Gconverse}.
In fact, the functor $\bft$ restricted to compacts identifies with
\[
\RHom_{\op A}(-,A)\colon \dbcat{\op A}\lra \op{\dbcat A}
\]
and this is an equivalence, again by
\intref{Theorem}{thm:Gconverse}. Thus $\bft$ is an equivalence of
categories. Moreover, since $\bfh$ is its adjoint, the latter is the
quasi-inverse to $\bft$.

For  $X\in \KProj A$  the equivalence of categories  and \intref{Lemma}{lem:AE} yield
\[
H^n(X) = \KHom(A,\Sigma^n X) \cong \KHom(E, \Sigma^n \bft X)  = H^n(\bft X)
\]
for each integer $n$. Thus $X$ is in $\KacProj A$ if and only $\bft X$ is in $\KacInj A$. Therefore $(\bft,\bfh)$ induce an equivalence on the subcategory of acyclic complexes.

The key to verifying the remaining assertions is the following

\begin{claim}
$\Inj A\subset \Loc(E)$, in $\KInj A$. 

\smallskip

Indeed, given the already established equivalence, it suffices to
verify that $\bfh I$ is in $\Loc(A)$ for any injective $A$-module $I$,
since $\bfh$ identifies $E$ with $A$. As $E$ is a bounded below
complex of injective modules, $\Hom_A(E,I)$ is a bounded above complex
of flat modules, and it is quasi-isomorphic to $I$, by
\intref{Lemma}{lem:AE}. Therefore $\bfh I = \bff\Hom_A(E,I)$ is a
projective resolution of $I$; see
\cite[Theorem~2.7(2)]{Iyengar/Krause:2006}. Thus $\bfh I$ is in
$\Loc(A)$, as desired.
\end{claim}

Fix $Y\in\KacInj  A$. Then $\KHom(E, \Sigma^n  Y)=0$ for each integer $n$, so the claim yields  $\KHom(I,\Sigma^n Y)=0$ for $I\in \Inj A$ and integers $n$, that is to say, $Y$ is totally acyclic. Thus any acyclic complex of injective modules is totally acyclic.

Fix an acyclic complex $X$ in $\KProj A$. We want to verify that $X$
is totally acyclic, that is to say, $\KHom(X,-)=0$ on $\Add A$. Since
$\bft$ is an equivalence of categories, it suffices to verify that
$\KHom(\bft X,-)=0$ on $\Add \bft A$, that is to say, on $\Add
E$. However, $\bft X$ is also acyclic, by the already established part
of the result, and any complex in $\Add E$ is bounded below, and hence
K-injective. This implies the desired result.
\end{proof}

\section{Gorenstein projective modules}
\label{se:GP}
Let $A$ be a Gorenstein $R$-algebra.  An $A$-module $M$ is \emph{Gorenstein projective} (abbreviated to G-projective) if $M$ is a syzygy in a totally acyclic complex of projective modules, that is, $M\cong \Coker(d_X^{-1})$ for some $X$ in
$\bfK_{\mathrm{tac}}(\Prj A)$. Given \intref{Theorem}{thm:documenta},
one can drop ``totally" from the definition. We write $\GProj A$ for
the full subcategory of $\Mod A$ consisting of G-projectives, and
$\Gproj A$ for $\GProj A\cap \mod A$.

Starting from \intref{Theorem}{thm:documenta}, and also the results below, one can develop the theory of G-projective modules along the lines in \cite{Buchweitz:1986} but we shall be content with recording a few observations needed to prove the duality theorems in \intref{Section}{se:GD}.  All these are well-known when $A$ is Iwanaga-Gorenstein. 

\begin{lemma}
\label{lem:Gproj-props}
Let  $M$ be a G-projective  $A$-module. The following statements hold.
\begin{enumerate}[\quad\rm(1)]
\item
$M_\fp$ is G-projective as an $A_\fp$-module  for $\fp\in\Spec R$.
\item
$\Tor^A_i(\omega_{A/R},M)=0 =\Ext_A^i(\omega_{A/R},M)$ for $i\ge 1$.
\end{enumerate}
\end{lemma}

\begin{proof}
Evidently the localisation of an acyclic complex is acyclic, so (1) follows. 

(2) Since an $A$-module is zero if it is zero locally on $\Spec R$, given (1) and the finite generation of $\omega$, we can reduce the verification of (2) to the case when $R$ is local, and so assume the injective dimension of $A$ is finite. Let $I$ be the injective hull of the residue field of $R$, and set  $J\colonequals \Hom_R(\omega,I)$.

\begin{claim}
The $A$-module $J$ is a faithful injective, and has finite projective dimension.

\medskip

Indeed, as $I$ is a faithful injective $R$-module it follows by adjunction that the $A$-module $J$ is faithful and injective. Since $R$ is a Gorenstein local ring it has finite injective dimension, so  $I$ has finite projective dimension; that is to say, $I$ is in $\Thick(\Add R)$ in $\dcat R$. Since $\omega$ is a finite projective $R$-module $\Hom_R(\omega,-)$ is an exact functor on $\dcat A$, so we deduce that $J$ is in $\Thick(\Add \Hom_R(\omega,R))$ in $\dcat A$.  Finally, observe that $A\cong \Hom_R(\omega, R)$ as $A$-modules.

\end{claim}

 The claim and the hypothesis that $M$ is G-projective justify the equality below:
\[
\Hom_R(\Tor^A_i(\omega,M),I) \cong \Ext^i_A(M,J) = 0 \quad\text{for $i\ge 1$;}
\]
see also \eqref{eq:Ext-vanishing}. The isomorphism is standard adjunction. Since $I$ is a faithful injective, it follows that $\Tor^A_i(\omega,M)=0$ as desired. 

A similar argument settles the claim about the vanishing of Ext-modules.
\end{proof}

When $M$ is G-projective and $X\in\KacProj A$ is as above, the truncation $X_{\ges 0}$ is a projective resolution of $M$ and  the total acyclicity of $X$ implies
\begin{equation}
\label{eq:Ext-vanishing}
\Ext_A^i(M,P)=0\quad\text{for each projective module $P$ and $i\ge 1$.}
\end{equation}
Here is a partial converse.

\begin{lemma}
\label{le:Ext1=gp}
A finitely generated $A$-module $M$ satisfying $\Ext_A^i(M,A)=0$ for $i\ge 1$ is G-projective. Moreover, such a module is a syzygy in an acyclic complex of finitely generated projective $A$-modules.
\end{lemma}

\begin{proof}
It suffices to verify that  $\Ext_{\op A}^i(M^*,A)=0$ for $i\ge 1$, and that the biduality map $M\to M^{**}$ is bijective; given these, it is straightforward to construct an acyclic complex with $M$ as a syzygy. What is more, using resolutions of $M$ and $M^*$ by finitely generated projective modules, one can get an acyclic complex  consisting of finitely generated projective modules.  Since $M$ is finitely generated, and $A$ is a finite $R$-algebra, both the conditions in question can be checked locally on $\Spec R$. We may thus assume that $A$ is Iwananga-Gorenstein, in which case the desired result is contained in \cite[Lemma~4.2.2(iii)]{Buchweitz:1986}.
\end{proof}

With exact structure inherited from $\Mod A$, the category $\GProj A$ is Frobenius, with projective objects $\Prj A$. Thus the associated stable category, $\uGProj A$, is triangulated. It is also compactly generated, with compact objects $\uGproj A$; see, for example, \cite[Proposition~2.10]{Benson/Iyengar/Krause/Pevtsova:2020a}. By the very definition, G-projectives are related to acyclic complexes of
projectives. To clarify this connection, we recall from \cite[\S7.6]{Iyengar/Krause:2006} that there is an adjoint pair
\[
\begin{tikzcd}
\KacProj A \arrow[hookrightarrow,yshift=.75ex]{rr}
		 &&\KProj  A \arrow[twoheadrightarrow,yshift=-.75ex]{ll}{\bfa}
\end{tikzcd}
\]
where the left adjoint is the inclusion. The next result is well known, and can be readily proved by adapting the argument for \cite[Theorem~4.4.1]{Buchweitz:1986}.

\begin{proposition}
\label{prop:GP=Kprj}
The composition of functors
$\bfa\circ\pres\colon \Mod A\to \KacProj A$ induces a triangle
equivalence
\[
\bfa\pres\colon \uGProj A \longiso \KacProj A\,,
\]
with quasi-inverse  defined by the assignment $X\mapsto \Coker(d^{-1}_X)$. \qed
\end{proposition}

\subsection*{The singularity category}
Let $\dsing(A)$ be the \emph{singularity category} of $A$, introduced by Buchweitz~\cite{Buchweitz:1986} as the \emph{stable derived category}. It is $\dbcat A$ modulo the perfect complexes.  Any perfect complex is in the kernel of the functor
\[
\bfs \ires\colon \dbcat A\to {\KacInj A}^{\mathrm{c}}
\]
where the functors $\bfs$ and $\ires$ are from \eqref{eq:recollement}. Hence there is an induced exact functor 
\[
\dsing(A)\to {\KacInj A}^{\mathrm{c}}
\]
that we also denote $\bfs\ires$. On the other side, the embedding $\Gproj A\hookrightarrow \dbcat A$ induces an exact functor
\[
\bfg \colon \uGproj A \longrightarrow \dsing(A)\,.
\]
The result below was proved by Buchweitz~\cite[Theorem~4.4.1]{Buchweitz:1986} when $A$ is an Iwanaga-Gorenstein ring.

\begin{theorem}
\label{thm:rob}
Let $A$ be a Gorenstein $R$-algebra. The functors $\bfg$ and $\bfs\ires$ are equivalences, up to direct summands, of triangulated
categories:
\[
\begin{tikzcd}
\uGproj(A)   \arrow{r}{\sim}[swap]{\bfg} & \dsing(A)   \arrow{r}{\sim}[swap]{\bfs\ires} &  {\KacInj A}^{\mathrm c}\,.
\end{tikzcd}
\]
\end{theorem}

\begin{proof}
The assertion about $\bfs\ires$ is by \cite[Corollary~5.4]{Krause:2005}. 

Let $M,N$ be finitely generated G-projective $A$-modules. As noted in \eqref{eq:Ext-vanishing}, one has $\Ext^i_A(M,A)=0$ for $i\ge 1$. Arguing as in the proof of \cite[Proposition~1.21]{Orlov:2004a} one gets that  $\bfg$ induces a bijection:
\[
\sHom_A(M,N) \xra{\cong} \Hom_{\dsing}(\bfg M,\bfg N)\,.
 \]
Thus $\bfg$ is fully faithful. It remains to prove that it  is essentially surjective. 

Fix $X$ in $\dsing(A)$; we can assume that $X$ is a bounded above complex of finitely generated projective $A$-modules. Suppose $H^i(X)=0$ for all $i < n$. Truncating at $n$ yields a morphism $X\to \sigma_{\le n}X$ which is an isomorphism in $\dsing(A)$ since its cone is perfect.  Thus $X$ is isomorphic to a suspension of $M\colonequals \Coker(d_X^{n-1})$ in $\dsing(A)$. Since $\Ext^i_A(M,A)$ for $i\gg 0$ by \intref{Theorem}{thm:Gconverse},  some syzygy of $M$ is G-projective by \intref{Lemma}{le:Ext1=gp}, and we conclude that $\bfg$ is essentially surjective.
\end{proof}

A standard d\'evissage argument yields the following consequence.

\begin{corollary}
\label{cor:GP=KInj}
The composition of functor $\bfs\circ \ires \colon \Mod A\to \KInj A$
induces a triangle equivalence
\[
\bfs\ires\colon \uGProj A \longiso \KacInj A\,.
\]
\end{corollary}

\begin{proof}
  The triangulated categories $\uGProj A$ and $\KacInj$ are both
  compactly generated and the functor $\bfs\ires$ preserves
  coproducts. For the compact generation of $\KacInj$, see
  \cite[Corollary~5.4]{Krause:2005}, and $\bfs\ires$ preserves coproducts
  since $\bfs$ is a left adjoint.  Thus the assertion follows from the
  fact that $\bfs\ires$ is a triangle equivalence when restricted to
  the subcategories of compact objects; see \intref{Theorem}{thm:rob}.
\end{proof}

\subsection*{The Nakayama functor}
Via the equivalences of categories established above the
auto-equivalence of $\KacInj A$ given by Nakayama functor induces an
auto-equivalence on $\uGProj A$ and on the singularity category. This
is made explicit in the next two results. The functor $\gp(-)$ that
appears in the statements is the G-projective approximations whose
existence is established in \intref{Theorem}{thm:g-approximations}.
When $A$ is a Gorenstein $R$-algebra, it follows from
\intref{Theorem}{thm:Morita-D} that functor $\omega_{A/R}$ takes
perfect complexes to perfect complexes, and hence induces a functor on
the quotient $\dsing(A)$; we denote that functor also
$\omega_{A/R}\lotimes_A(-)$.

\begin{proposition}
\label{prp:Nak-G-sing}
Let $A$ be a Gorenstein $R$-algebra. One has the following diagram of equivalences of categories
\[
\begin{tikzcd}
\uGproj(A)   \arrow{d}{\sim}[swap]{\gp(\omega_{A/R}\otimes_A(-))}\arrow{r}{\sim}[swap]{\bfg} 
	& \dsing(A)   \arrow{r}{\sim}[swap]{\bfs\ires} \arrow{d}{\sim}[swap]{\omega_{A/R}\lotimes_A(-)}
		&  {\KacInj A}^{\mathrm c} \arrow{d}{\wh\kan_{A/R}^{-1}}[swap]{\sim} \\
\uGproj(A)	\arrow{r}{\sim}[swap]{\bfg} & \dsing(A)   \arrow{r}{\sim}[swap]{\bfs\ires} &  {\KacInj A}^{\mathrm c} 
\end{tikzcd}
\]
where the squares commute up to an isomorphism of functors. 
\end{proposition}

\begin{proof}
The equivalences in the rows are from \intref{Theorem}{thm:rob}. We already know that $\wh{\kan}$ is an equivalence, so one has only to verify the commutativity of the diagram. 

The commutativity of the square on the left is tantamount to: For each
G-projective $A$-module $M$ there is a natural isomorphism between
$\omega\lotimes_AM$ and $\gp(\omega\otimes_AM)$, viewed as objects in
$\dsing(A)$. As noted in \intref{Lemma}{le:Ext1=gp}, finitely
generated G-projective modules are syzygies in acyclic complexes of
finitely generated projective modules. Thus the proof of \intref{Theorem}{thm:g-approximations} yields
an exact sequence of $A$-modules
\[
0\lra P \lra \gp(\omega\otimes_AM) \lra \omega\otimes_A M\lra 0
\]
with $\gp(\omega\otimes_AM)$ a G-projective and $P$ a finitely generated projective.   This gives the isomorphism on the left
\[
\gp(\omega\otimes_AM) \longiso \omega\otimes_A M \longliso \omega\lotimes_AM
\]
in $\dsing(A)$. The one on the right is by \intref{Lemma}{lem:Gproj-props}(2), for the latter is tantamount to the statement that the natural morphism of complexes $\omega\lotimes_AM\to (\omega\otimes_AM)$ is an isomorphism in $\dcat A$, and so also in $\dsing A$. 

For $X\in\dbcat A$, from \intref{Lemma}{lem:cNak} and \intref{Theorem}{thm:Morita-D} one gets isomorphisms
\[
\wh\kan \ires(\omega\lotimes_A X) \cong \wh\kan (\omega\lotimes_A\ires X) \cong \RHom_A(\omega,\omega\lotimes_A\ires X) \cong \ires X\,.
\]
Applying $\bfs$ to the composition, and observing that $\wh\kan$ commutes with $\bfs$ by \intref{Theorem}{thm:Morita-K}, 
yields the  commutativity of the square on the right.
 \end{proof}

The commutativity of the outer square in \intref{Proposition}{prp:Nak-G-sing} lifts to the corresponding ``big" categories.

\begin{proposition}
\label{prp:g=tac}
The functor $\gp(\omega_{A/R}\otimes_A-)\colon \uGProj A\to \uGProj A$ is an equivalence of triangulated categories, with quasi-inverse $\gp\Hom_A(\omega_{A/R},-)$. Moreover, the diagram below commutes up to an isomorphism of functors:
\[
\begin{tikzcd}[column sep=large]
\uGProj A  \arrow[rightarrow]{d}{\sim}[swap]{\gp_A(\omega_{A/R}\otimes_A-)} \arrow[rightarrow]{r}{\bfs\ires}[swap]{\sim} 
	& \KacInj A \arrow[rightarrow]{d}{{\wh\kan_{A/R}}^{-1}}[swap]{\sim} \\
\uGProj A \arrow[rightarrow]{r}{\bfs\ires}[swap]{\sim}  & \KacInj A\,.
\end{tikzcd}
\]
\end{proposition}

\begin{proof}
The crucial observation is that the categories involved are compactly generated, and all the functors involved commute with direct sums. Thus the desired result is a consequence of \intref{Proposition}{prp:Nak-G-sing}.
\end{proof}

\section{Localisation and torsion functors}
\label{se:bik}
As before let $A$ be a finite $R$-algebra. In what follows we apply the theory of local cohomology and localisation from \cite{Benson/Iyengar/Krause:2008a}, with respect to the action of the ring $R$ on the homotopy category of injective $A$-modules. To that end we recall some results concerning the structure of injective $A$-modules discovered by Gabriel~\cite{Gabriel:1962}; it extends the (by now well-known) theory for commutative  rings.

To begin with, by the \emph{spectrum} of $A$ we mean the collection of two sided prime ideals of $A$, denoted $\Spec A$. Since the map $\eta\colon R\to A$ is central and finite, the induced map on spectra
\[
\Spec A\lra \Spec R \quad \text{where $\fq\mapsto\fq\cap R$ for $\fq\in\Spec A$},
\]
is surjective onto $\Spec \eta(R)$, which is a closed subset of
$\Spec R$. Moreover, the fibres of the map are \emph{discrete}: if
$\fq'\subseteq \fq$ are elements of $\Spec A$ such that
$\fq'\cap R=\fq\cap R$, then $\fq'=\fq$; see \cite[Proposition~V.11]{Gabriel:1962}.

\subsection*{Torsion}
For each $\fp$ in $\Spec R$ there is a natural $A$-module structure of $M_\fp$ for which the canonical map $M\to M_\fp$ is $A$-linear.

A subset $V\subseteq \Spec R$ is \emph{specialisation closed} when it
has the following property: If $\fp\subseteq\fp'$ are prime ideals in
$R$ and $\fp$ is in $V$, then $\fp'$ is in $V$; equivalently, that $V$
contains the closure (in the Zariski) topology of its points. The
following specialisation closed subsets play a central role: Given
an ideal $\mathfrak a\subset R$, the subset
\[
V(\mathfrak a)\colonequals \{\fp\in\Spec R\mid \fp\supseteq \mathfrak a\}
\]
of $\Spec R$, and  given a prime $\fp$ in $\Spec R$, the subset
\[
Z(\fp)\colonequals \{\fp'\in\Spec R\mid \fp'\not\subseteq \fp\}\,.
\]
Observe that $Z(\fp)$ equals $\Spec R\setminus \Spec R_\fp$.

Give a specialisation closed subset $V$ of $\Spec R$, the
\emph{$V$-torsion} submodule of an $A$-module $M$ is defined by
\[
\gam_VM\colonequals \Ker(M\lra \prod_{\fp\not\in V}M_\fp)\,.
\]
The assignment $M\mapsto \gam_VM$ is an additive, left-exact, functor
on $\Mod A$. The module $M$ is called \emph{$V$-torsion} if
$\gam_V M=M$.

It is easy to verify that when $V\colonequals V(r)$ for an element $r\in R$, one has
\[
\gam_{V(r)}M = \Ker(M\lra M_r)
\]
where $M_r$ is the localisation of $M$ at the multiplicatively closed subset $\{r^n\}_{n\ges 0}$, and that when $V\colonequals Z(\fp)$, for some $\fp\in \Spec R$, one gets
\[
\gam_{Z(\fp)}M = \Ker(M\lra M_\fp)\,.
\]

\subsection*{Injective modules}
Since $A$ is noetherian, $\Inj A$, the full subcategory of $\Mod A$
consisting of injective modules, is closed under arbitrary direct
sums. For a $\fq$ in $\Spec A$ the injective hull of the $A$-module
$A/\fq$ decomposes into a finite direct sum of copies of an
indecomposable injective module which we denote by $I(\fq)$. Since $A$
is a finite $R$-algebra, the assignment $\fq\mapsto I(\fq)$ is a
bijection between $\Spec A$ and the isomorphism classes of
indecomposable injective $A$-modules, by
\cite[V.4]{Gabriel:1962}. Thus each injective $A$-module is a direct
sum of copies of $I(\fq)$, as $\fq$ varies over $\Spec A$.

\begin{lemma} 
\label{lem:gamE}
Let $V\subseteq \Spec R$ be specialisation closed. For each injective $A$-module $I$, the module $\gam_V I$ is a
direct summand of $I$. Thus for $\fq$ in $\Spec A$ one has
\[
\gam_VI(\fq)=\begin{cases}
I(\fq) & \text{when $\fq\cap R\in V$},\\
0 &\text{otherwise}.
\end{cases}
\]
 \end{lemma}
 
 \begin{proof}
   The functor $\gam_V$ provides a right adjoint for the inclusion of
   the localising subcategory of $A$-modules that are
   $V$-torsion. This functor preserves injectivity, since the
   localising subcategory is stable under taking injective envelopes,
   by \cite[Proposition~V.12]{Gabriel:1962}. Thus $\gam_V I$ is a
   direct summand of $I$ for every injective $A$-module $I$. In
   particular, we have $\gam_V I= I$ or $\gam_V I= 0$ when $I$ is 
   indecomposable.
 \end{proof}
 
 Since $\gam_V$ is an additive functor, it induces a functor on the
 category of $A$-complexes. For each complex $X$ of injective
 $A$-modules set
\[
L_VX\colonequals \Coker(\gam_VX\lra X)\,.
\]
Thus one gets an exact sequence of $A$-complexes
\[
0 \lra \gam_VX\lra X \lra L_VX\lra 0\,.
\]
By \intref{Lemma}{lem:gamE} the subcomplex $\gam_VX$ consists of injective $A$-modules so the sequence above is degree-wise split
exact, and hence induces in $\KInj A$ an exact triangle
\begin{equation}
\label{eq:locV}
\gam_VX\lra X \lra L_VX\lra \Sigma\gam_VX\,.
\end{equation}

The functor $L_V$ has an explicit description in a couple of cases.

\begin{example}
\label{exa:vr}
Suppose $V\colonequals V(r)$, for some $r\in R$. Then the map $X\to X_r$ is surjective, by \intref{Lemma}{lem:gamE}, so there is an exact sequence
\[
0 \lra \gam_{V(r)} X\lra X \lra X_r\lra 0
\]
 of $A$-complexes, and hence $L_{V(r)}X = X_r$.  By the same token, when  $V\colonequals Z(\fp)$ for some prime $\fp$ in $\Spec R$, one gets an exact sequence
\[
0 \lra \gam_{Z(\fp)} X\lra X \lra X_\fp\lra 0
\]
of $A$-complexes, so that $L_{Z(\fp)}X = X_\fp$.
 \end{example}

 \subsection*{Localisation and local cohomology}
 
The ring $R$ acts on $\KMod A$ and hence on its subcategories discussed above, in the sense of \cite{Benson/Iyengar/Krause:2008a}. We focus on $\cat T=\KInj A$.

For any localising subcategory $\cat C\subseteq \cat T$ and object  $X\in\cat T$ we call an exact triangle
\[
\gam X\lra X \lra L X\lra \Sigma\gam X
\]
a \emph{localisation triangle} provided that $\gam X\in \cat C$ and $LX\in\cat C^\perp$, where $\cat C^\perp\subseteq\cat T$ denotes the colocalising subcategory consisting of objects $Y$ such that $\Hom_{\cat T}(X,Y)=0$ for all $X\in\cat C$. If such a triangle exists for all objects $X\in\cat T$, then $\gam$ yields a right adjoint for the inclusion $\cat C\hookrightarrow\cat T$ and $L$ yields a left adjoint for the  inclusion $\cat C^\perp\hookrightarrow\cat T$.

Given a specialisiation closed subset $V\subseteq\Spec R$, an object $X$ in $\cat T$ is  \emph{$V$-torsion} provided that $\Hom_{\cat T}(C,X)$
is a $V$-torsion $A$-module for each compact $C\in\cat T$.

\begin{lemma}
  For a specialisiation closed subset $V\subseteq\Spec R$, the  triangle  \eqref{eq:locV} is the localisation triangle associated to the localising subcategory of $V$-torsion objects in $\KInj A$.
\end{lemma}

\begin{proof}
Fix $X\in\KInj A$. Then $\gam_VX$ is $V$-torsion by construction. Moreover, for every injective $A$-module $I$ it is easy to verify that
\[
\Hom(\gam_V I,I/\gam_V I)=0\,.
\] 
Thus $\KHom(X',L_V X)=0$ for all  $V$-torsion $X'\in\KInj A$.
\end{proof}

\begin{lemma}
  For any $\fp$ in $\Spec R$ and $X$ in $\KInj A$ we have 
  $L_{Z(\fp)}X\cong X_\fp$.
\end{lemma}
\begin{proof}
This follows from \intref{Example}{exa:vr}.
\end{proof}

For an object $X$ in $\KInj A$, we write $\Loc(X)$ for the smallest
localising subcategory of $\KInj A$ that contains $X$.

\begin{lemma}
\label{lem:lgp}
Let $V\subseteq \Spec R$ be specialisation closed. For any $X$ in
$\KInj A$, the $A$-complexes $\gam_VX$ and $L_VX$ are in
$\Loc(X)$.
\end{lemma}

\begin{proof}
This follows from the local-to-global principle discussed in \cite{Benson/Iyengar/Krause:2011a}. More specifically, one combines   \cite[Theorem~3.1]{Benson/Iyengar/Krause:2011a} with \cite[Theorem~6.9]{Stevenson:2011}.
\end{proof}

 \begin{lemma}
 \label{lem:lipman}
 Let $V\subseteq \Spec R$ be specialisation closed. If an $A$-complex $X$ of injective $A$-modules is acyclic, then so are the complexes
 $\gam_VX$ and $L_VX$.
 \end{lemma}
 
 \begin{proof}
   The subcategory $\KacInj A$ of $\KInj A$ is localising, hence when  $X$ is acyclic so are the complexes in $\Loc(X)$. It remains to recall \intref{Lemma}{lem:lgp}.
 \end{proof}

For an object $X$ in $\KInj A$ and $\fp\in\Spec R$ the \emph{local cohomology}  at $\fp$ is
\[
\gam_\fp X:=\gam_{V(\fp)}(X_{\fp})\,.
\]
The following observation will be useful.

\begin{lemma}
\label{lem:bb}
For any $X$ in $\KInj A$, the complex $\gam_{\fp}X$ is a subquotient
of $X$. In particular, if $X^{i}=0$ for some $i\in\mathbb Z$ then $(\gam_{\fp}X)^{i}=0$  as well. \qed
\end{lemma}

\begin{remark}
\label{rem:kac}
The triangulated category $\KacInj A$ is compactly generated and $R$-linear, so has its own localisation functors for a specialisation closed subset $V$ of $\Spec R$. It follows from \intref{Lemma}{lem:lipman} that these are just restrictions of the corresponding functors on  $\KInj A$. 

The triangulated category $\dcat A$ is also compactly generated and $R$-linear. However the embedding $\ires\colon \dcat A\to \KInj A$ is not compatible with the localisation functors; in other words, for a K-injective complex $X$,  the complex $\gam_VX$ need not be K-injective; see \cite{Chen/Iyengar:2010}.
On the other hand, it is easy to verify that these functors are compatible with the restriction functor $\dcat A\to \dcat R$.
\end{remark}

\begin{remark}
\label{rem:s-local}
Fix a $\fp$ in $\Spec R$ and consider the diagram of exact functors.
\begin{equation*}
\begin{tikzcd}[row sep=large]
  \KacInj A \arrow[hookrightarrow,yshift=-.75ex,swap]{rr}{\mathrm{incl}}
  \arrow[twoheadrightarrow,xshift=-.75ex,swap]{d}{(-)_\fp}&& \KInj A
  \arrow[twoheadrightarrow,yshift=.75ex,swap]{ll}{\bfs}
    \arrow[twoheadrightarrow,xshift=-.75ex,swap]{d}{(-)_\fp}\\ 
  \KacInj{A_\fp}
  \arrow[hookrightarrow,yshift=-.75ex,swap]{rr}{\mathrm{incl}}
  \arrow[tail,xshift=.75ex,swap]{u}{\mathrm{res}}&& \KInj{A_\fp}
  \arrow[twoheadrightarrow,yshift=.75ex,swap]{ll}{\bfs_\fp}
  \arrow[tail,xshift=.75ex,swap]{u}{\mathrm{res}}
\end{tikzcd}
\end{equation*}
It is clear that the two compositions of right adjoints, from the bottom left to the top right, coincide. It follows that the composition of the corresponding left adjoint functors  are isomorphic: $(\bfs X)_\fp \cong \bfs_\fp(X_\fp)$ for  $X$ in $\KInj A$.

\end{remark}

\subsection*{Support}
Let  $\cat T$ be $\KInj A$ or $\KacInj A$. Specialising the definition from \cite{Benson/Iyengar/Krause:2008a} to our context, we introduce the \emph{support} of an object $X$ in $\cat T$ to be the subset
\[
\supp_R X\colonequals \{\fp\in\Spec R\mid \gam_\fp X\ne 0\}\,.
\]
It follows from \intref{Remark}{rem:kac} that the support an object in $\KacInj A$ is the same as its support when we view it as an object in $\KInj A$.

The \emph{support} of $\cat T$ is the subset of $\Spec R$ defined by
\[
\supp_R\cat T\colonequals \bigcup_{X\in{\cat T}^{\mathrm c}} \supp_R X\,.
\]
Here are some alternative characterisations of  support for acyclic complexes.

\begin{proposition}
  Let $A$ be a finite $R$-algebra, fix $X\in \KacInj A$ and $\fp$ in
  $\Spec R$. The following conditions are equivalent:
\begin{enumerate}[\quad\rm(1)]
\item
The prime $\fp$ is not in $\supp_RX$.
\item
The complex $\gam_\fp X$ is contractible.
\item
The $A$-module $\gam_\fp(\Omega^i(X))$ is injective for each (equivalently, some) integer $i$.
\end{enumerate}
\end{proposition}

\begin{proof}
An acyclic complex of injective modules is zero in $\KInj A$ if and only if it is contractible, if and only if each, equivalently, one of its syzygy modules is injective. From this we get that (1)$\Leftrightarrow$(2) and also that these conditions are equivalent to $\Omega^i(\gam_\fp X)$ injective for each, equivalently, some, $i$. It remains to note that since the functor $\gam_\fp$ is left-exact, and preserves acyclicity of complexes in $\KInj A$, one gets
\[
\Omega^i(\gam_\fp X) \cong \gam_\fp \Omega^i( X) \quad\text{for each integer $i$.}
\]
This completes the proof.
\end{proof}

The following observation concerning generators for $\KacInj A$ is well-known.

\begin{lemma}
\label{le:kac-gen}
The compact objects in $\KacInj A$ are direct summands of objects of the form $\bfs C$, where $C$ is a compact object in $\KInj A$.
\end{lemma}

\begin{proof}
  The functor $\bfs$ is left adjoint to the inclusion
  $\KacInj A\subset \KInj A$, so it is essentially surjective; it also
  preserves compactness for the inclusion preserves direct sums. It
  follows that up to direct summands all compact objects of
  $\KacInj A$ are in the image of $\bfs$; see
  \cite[Theorem~2.1]{Neeman:1992}.
\end{proof}

A noetherian ring $A$ is \emph{regular} if each $M\in \mod A$ has finite projective dimension; equivalently, each $M$ in $\dbcat A$ is perfect.  We say that $A$ is \emph{singular} to mean that it is not regular.   When $A$ is a finite $R$-algebra its \emph{regular locus} will mean the collection of primes $\fp\in\Spec R$ such that $A_\fp$ is regular. It's complement in $\Spec R$ is the \emph{singular locus}.

\begin{corollary}
\label{cor:suppKacinj}
The  singular locus of $A$ equals $\supp_R\KacInj A$. 
\end{corollary}

\begin{proof}
By \intref{Lemma}{le:kac-gen} the support of $\KacInj A$ is the union of the supports of $\bfs(\ires M)$, for $M\in \dbcat A$. For any $\fp\in \Spec R$  one has isomorphisms
\[
\bfs(\ires M)_\fp \cong \bfs_\fp( (\ires M)_\fp) \cong \bfs_\fp(\ires (M_\fp))
\]
in $\KInj{A_\fp}$, where the first one is by \intref{Remark}{rem:s-local} and the second one is standard. Thus $\bfs(\ires M)_\fp\cong 0$ if and only if $M_\fp$ is perfect in $\dcat{A_\fp}$.  Consequently, if $\fp$ is in the regular locus of $A$, then $\bfs(\ires M)_\fp=0$ for each  $M$ in $\dbcat A$, and hence $\fp$ is not in the support of $\KacInj A$.

Conversely, if $A_\fp$ is not regular, then there exists an $M\in \mod A$ such that $M_\fp$ is not perfect; one can choose $M$ to be $V(\fp)$-torsion. Then $\gam_{\fp} \bfs(\ires M) \cong \bfs(\ires M)_\fp$ is nonzero, so $\fp$ is in the support of $\KacInj A$.
\end{proof}

\section{Matlis duality and Gorenstein categories}
This section is about avatars of Matlis duality in various homotopy categories we have been dealing with.  To set the stage for the discussion, it helps to consider a general, compactly generated, triangulated category $\cat T$ with the action of a commutative noetherian ring $R$, in the sense of \cite{Benson/Iyengar/Krause:2008a}.  Fix an injective $R$-module $I$. For each compact object $C$ in $\cat T$ the functor
\[
X\longmapsto \Hom_{R}(\Hom_{\cat T}(C,X),I)
\]
from $\cat T$ to $\Mod R$ is homological and takes coproducts to products. The Brown Representability Theorem implies that it is representable: There is an object, say $T_I(C)$, in $\cat T$ and an isomorphism of functors 
\[
\Hom_R(\Hom_{\cat T}(C,-),I) \cong \Hom_{\cat T}(-,T_I(C))\,.
\]
 In this way the  assignment $C \times I \mapsto T_I(C)$ yields a functor
\[
T \colon  {\cat T}^{\mathrm c} \times \Inj R \lra {\cat T}\,.
\]

Borrowing terminology from \cite{Dwyer/Greenlees/Iyengar:2005} we call the functor $T_I(-)$ the \emph{Matlis lift} of $I$ to $\cat T$. In what follows, for $\fp$ in $\Spec R$ we write $T_{\fp}(-)$ for  $T_{I(\fp)}(-)$, where $I(\fp)$ is the injective hull of the $R$-module $R/\fp$.

Let now $A$ be a finite $R$-algebra as before. The description of the Matlis lifts of
injective $R$-modules to the $R$-linear category $\dcat A$ is straightforward.

\begin{proposition}
\label{pro:Matlis-D}
The Matlis lift to $\dcat A$ of an injective $R$-module $I$ is given
by the functor $C\mapsto \RHom_R(A,I)\lotimes_AC$.
\end{proposition}

\begin{proof}
  Given objects $X\in\dcat A$ and a  finitely generated projective $A$-module $P$, there are natural isomorphisms
  \begin{align*}
    \Hom_R(\Hom_A(P,X),I)&\cong\Hom_R(X,I)\otimes_A P\\
                         &\cong\Hom_A(X,\Hom_R(A,I))\otimes_A P\\
                         &\cong\Hom_A(X, \Hom_R(A,I)\otimes_A P)\,.
    \end{align*}
It remains to observe that any compact object in $\dcat A$ is isomorphic to a bounded complex of finitely generated projective $A$-modules.
\end{proof}

The Matlis lifts of
injective $R$-modules to the $R$-linear category $\KInj A$ is
described in the next result, which is modeled on
\cite[Theorem~3.4]{Krause/Le:2006a}; the proof we give is somewhat
different.

\begin{theorem}
\label{thm:Matlis}
Let $A$ be a finite $R$-algebra. The Matlis lift to $\KInj A$ of an
injective $R$-module $I$ is given by
\[
C\longmapsto  \Hom_R(A,I)\otimes_A \pres  C\,.
\]
\end{theorem}

\begin{proof}
Fix objects $C,X$ in $\KInj A$ with $C$ compact. The key input is \intref{Lemma}{lem:an-iso} that yields the first isomorphism below
\begin{align*}
\Hom_{R}(\KHom(C,X),I) 
	&\cong \Hom_{R}(H^{0}(\Hom_{A}(\pres C,A)\otimes_{A}X), I) \\
	&\cong H^0(\Hom_R(\Hom_{A}(\pres C,A)\otimes_{A}X,I)) \\
	&\cong H^0(\Hom_A(X, \Hom_{R}(\Hom_{A}(\pres C,A),I)) \\
        &\cong H^0(\Hom_A(X, \Hom_R(A,I) \otimes_{A} \pres C) )\\
	&\cong \KHom(X, \Hom_R(A,I) \otimes_{A} \pres C)\,.
\end{align*}
The second one holds because $I$ is injective. The rest are standard.
\end{proof}

The next result describes Matlis lifts to $\KacInj A$, using the
functors from \eqref{eq:recollement}. In \intref{Lemma}{le:kac-gen} we
described the compact objects in that category.

\begin{corollary}
\label{cor:TsT}
For a compact object in  $\KacInj A$ of the form $\bfs C$,  given by a
compact object $C$ in $\KInj A$, the Matlis lift of an injective
$R$-module $I$ is the complex
\[
T_{I}(\bfs C) \cong \bfr (T_{I} C ) \cong \bfr (\Hom_R(A,I)\otimes_A \pres  C)\,.
\]
\end{corollary}

\begin{proof}
For any acyclic complex $X$ of injective $R$-modules, one has
\begin{align*}
\KHom (X,\bfr(T_{I}(C)) 
	& \cong \KHom (X,T_{I}(C)) \\
	& \cong \Hom_{R}(\KHom (C,X), I) \\
	& \cong \Hom_{R}(\KHom (\bfs C,X), I)\\
   	& \cong \KHom (X,T_{I}(\bfs C))\,.
\end{align*}
This justifies the first isomorphism. For the second one, see \intref{Theorem}{thm:Matlis}.
\end{proof}

\begin{remark}
There is a notion of purity for compactly generated triangulated  categories, analogous to the classical concept of purity for module categories; see Crawley-Boevey's survey \cite{CB1998}. It follows from the construction that any Matlis lift is a pure-injective object. In particular, we obtain from a Matlis lift a pure-injective module when an acyclic complex is identified with an $A$-module.
\end{remark}

\subsection*{Gorenstein categories}
Let $\cat T$ be an $R$-linear category. Following \cite{Benson/Iyengar/Krause/Pevtsova:2019a} we say that $\cat T$ is \emph{Gorenstein} if there is an $R$-linear triangle equivalence
\[ 
F\col {\cat T}^{\mathrm{c}}\longiso {\cat T}^{\mathrm{c}}
\] 
such that for each $\fp$ in $\supp_{R} \cat T$ there is an integer $d(\fp)$ and a natural isomorphism
\[ 
 \gam_\fp\comp F \cong   \Sigma^{-d(\fp)}\circ T_\fp
\] 
of functors $\cat T^{\mathrm{c}}\to \cat T$. The functor $F$ plays the role of a \emph{global Serre functor} because it induces a Serre functor, in the sense of Bondal and Kapranov~\cite{Bondal/Kapranov:1990a}, on the subcategory of compacts objects in ${\cat T}_\fp$, the $\fp$-local $\fp$-torsion objects in $\cat T$, for  $\fp$ in $\Spec R$. More precisely, localising with respect to $\fp$ yields a functor $F_\fp\col {\cat T}_\fp^{\mathrm{c}}\iso {\cat T}_\fp^{\mathrm{c}}$ and a natural isomorphism
\[
\Hom_R(\Hom_{\cat T}(X,Y),I(\fp))\cong\Hom_{\cat T}(Y,\Sigma^{d(\fp)}F_\fp X)
\]
for objects $X,Y\in\cat T_\fp$ such that $X$ is compact and $\supp_R X=\{\fp\}$.  This is explained in \cite[\S7]{Benson/Iyengar/Krause/Pevtsova:2019a}. In what follows we focus on the following special case.

\begin{proposition}
\label{pr:Serre}
Let $\cat T$ be a compactly generated $R$-linear category that is Gorenstein, with global Serre functor $F$. Fix a maximal ideal $\fm$
in $R$.  For any $X\in {\cat T}^{\mathrm c}$ and $Y\in\cat T$ with $\supp_RX=\{\fm\}$ there is a natural isomorphism
\[
\Hom_R(\Hom_{\cat T}(X,Y), I(\fm)) \cong \Hom_{\cat T}(Y,\Sigma^{d(\fm)}FX) 
\]
In particular, if $\supp_R\cat T=\{\fm\}$, then $\Sigma^{d(\fm)}F$ is a Serre functor on ${\cat T}^{\mathrm c}$.
\end{proposition}

\begin{proof}
Since $\fm$ is maximal, any object supported on $\fm$ is already $\fm$-local. Thus the desired isomorphism is a special case of \cite[Proposition~7.3]{Benson/Iyengar/Krause/Pevtsova:2019a}. 
\end{proof}

\subsection*{Gorenstein rings}
Let $R$ be a commutative Gorenstein ring. For  $\fp$ in $\Spec R$, set $h(\fp)=\dim R_{\fp}$; this is the height of $\fp$. The Gorenstein property for $R$ is equivalent to the condition that the minimal injective resolution $I$ of $R$ satisfies
\[
I^{n} = \bigoplus_{h(\fp) = n} I(\fp)  \quad\text{for each $n$.}
\]
This translates to the condition that in $\KInj R$ there are isomorphisms
\begin{equation}
\label{eq:gd-commutative}
\gam_{\fp}(\ires R) \cong \Sigma^{-h(\fp)} I(\fp) \quad\text{for each $\fp\in \Spec R$.}
\end{equation}
This result is due to Grothendieck, cf.\ \cite[Proposition~3.5.4]{Bruns/Herzog:1998a}.

\begin{proposition}
\label{prp:Gor-D}
Let $A$ be a finite $R$-algebra that is projective as an $R$-module. The following conditions are equivalent:
\begin{enumerate}[\quad\rm(1)]
\item
The $R$-algebra $A$ is Gorenstein.
\item
The $R$-linear category $\dcat A$ is Gorenstein.
\end{enumerate}
When they hold the global Serre functor is $\omega_{A/R}\lotimes_A-$, and $d(\fp)=\dim R_\fp$.
\end{proposition}

\begin{proof}
(1)$\Rightarrow$(2)
As the $R$-algebra $A$ is Gorenstein, the functor $F\colonequals \omega\lotimes_A-$ is an equivalence on $\dcat A$,  and hence restricts to an equivalence ${\dcat A}^{\mathrm c}$, the subcategory of perfect complexes; see \intref{Theorem}{thm:Morita-D}. With $d(\fp)$ as in the statement, for any perfect complex $C$, from \intref{Proposition}{pro:Matlis-D} one gets the equality below
\begin{align*}
T_{\fp}(C) 
	&= \RHom_R(A,I(\fp))\lotimes_A C \\
        &\cong I(\fp) \lotimes_R \Hom_R(A,R)  \lotimes_A C \\
        &\cong  \Sigma^{d(\fp)} \gam_\fp(\ires R) \lotimes_R FC \\
        & \cong \Sigma^{d(\fp)}\gam_{\fp}( \ires R \lotimes_RFC) \\
        & \cong \Sigma^{d(\fp)}\gam_{\fp}FC\,.
\end{align*}
The third isomorphism is from \eqref{eq:gd-commutative}, and the rest are standard. Thus $\dcat A$ is Gorenstein, with the prescribed global Serre functor and shift $d(\fp)$. 

(2)$\Rightarrow$(1) It suffices to verify that the injective dimension of $A_\fm$ is finite for any maximal ideal $\fm$ in $R$. For this it suffices to verify that   $M\in \mod(A/\fm A)$ satisfy
\[
\Ext^i_A(M,A)  =  0 \quad\text{for $i\gg 0$.}
\]
For then, an argument along the lines of the proof of \cite[Proposition~A.1.5]{Avramov/Iyengar:2019a} yields that $A_\fm$ has finite injective dimension over itself.

Let $F\colon {\dbcat A}^{\mathrm c}\to {\dbcat A}^{\mathrm c}$ be a global Serre functor, and $F^{-1}$ its quasi-inverse. Since $M$ is $\fm$-torsion from \intref{Proposition}{pr:Serre} we get the isomorphism below
\[
\DHom A(M, \Sigma^iA) \cong \Hom_R(\DHom A(F^{-1}A, \Sigma^{d(\fm)-i}M), I(\fm))\,.
\]
It remains to note that since $F^{-1}A$ is perfect one has
\[
\DHom A(F^{-1}A, \Sigma^j(-))=0 \quad \text{on $\Mod A$,}
\]
for all $|j|\gg 0$. This implies the desired result.
\end{proof}

Here is the analogue of the preceding result dealing with homotopy categories.

\begin{proposition}
\label{prp:Gor-K}
Let $A$ be a finite $R$-algebra that is projective as an $R$-module. The $R$-linear category $\KInj A$ is Gorenstein if and only if $A$ is regular. 
\end{proposition}

\begin{proof}
When $A$ is regular the canonical functor $\KInj A\to \dcat A$ is an equivalence and $\dcat A$ is Gorenstein, by \intref{Proposition}{prp:Gor-D}.  As to the converse, it suffices check that $A_\fm$ is regular for each maximal ideal $\fm$ in $R$. 

Arguing as in the proof of  (2)$\Rightarrow$(1) in \intref{Proposition}{prp:Gor-D} one deduces that for each $M\in \mod(A/\fm A)$ and $N\in\mod A$ one has
\[
\Ext^i_A(M,N)=0 \quad\text{for $i\gg 0$.}
\]
This implies that $A_\fm$ is regular.
\end{proof}

The preceding results concern the Gorenstein property for the derived category and the homotopy category of injectives for two of the three categories that appear in the recollement~\eqref{eq:recollement}. That of the last one is dealt with in the next section.

\section{Grothendieck duality for $\KacInj A$}
\label{se:GD}
This section is dedicated to the proof of the following result. As explained in the introduction, this has been the guiding light of the results presented in this work.

\begin{theorem}
\label{thm:Gor}
Let $A$ be a Gorenstein $R$-algebra. For each compact object $X$ in $\KacInj A$ and $\fp$ in the singular locus of $A$ there is is a natural isomorphism
\[
\gam_{\fp} X \cong \Sigma^{-d(\fp)} T_{\fp} (\wh\kan_{A/R} X) 
\]
where $d(\fp)= \dim (R_\fp) - 1$. In particular, the $R$-linear category $\KacInj A$ is Gorenstein, with global Serre functor the quasi-inverse of $\wh\kan_{A/R}$.
\end{theorem}

The proof is given further below. \intref{Theorem}{ithm:Gor} from the
introduction is an immediate consequence.

\begin{corollary}
  Let $A$ be a Gorenstein $R$-algebra, and let $M,N$ be G-projective
  $A$-modules with $M$ finitely generated.  For each $\fp\in\Spec R$
  there is a natural isomorphism
\[
\Hom_R(\tExt_A^i(M,N),I(\fp))\cong \tExt_A^{d(\fp)-i}(N, \gam_\fp S(M)).
\]
\end{corollary}
\begin{proof}
The assertion is a direct translation of \intref{Theorem}{thm:Gor}, given
the equivalence  $\uGProj A\iso \KacInj A$ from \intref{Proposition}{prp:g=tac}.
\end{proof}

We continue with a consequence concerning duality for the category of
compact objects. The statements are simpler, and perhaps more
striking, when specialised to the case of local isolated
singularities, and that is what we do.

\subsection*{Isolated singularities}
Let $(R,\fm,k)$ be commutative noetherian local ring and $A$ a finite projective $R$-algebra. We say that $A$ has an \emph{isolated singularity} if its singular locus is $\{\fm\}$; that is to say,  if the the ring $A_\fp$ is regular for each non-maximal ideal $\fp$ in $\Spec R$; see the discussion around \intref{Corollary}{cor:suppKacinj}. 

\begin{corollary}
\label{cor:serre}
Let $R$ be a commutative noetherian local ring  of Krull dimension $d$. If $A$ is a Gorenstein $R$-algebra with an isolated singularity, then the assignment
\[
X\mapsto  \Sigma^{d-1} \wh\kan^{-1}(X) 
\]
is a Serre functor on the $R$-linear category ${\KacInj A}^{\mathrm c}$.
\end{corollary}

\begin{proof}
Since $A$ has an isolated singularity, the $R$-linear category $\KacInj A$ is supported at $\fm$, the maximal ideal of $R$; see \intref{Corollary}{cor:suppKacinj}.
Thus \intref{Theorem}{thm:Gor} and \intref{Proposition}{pr:Serre} yield the desired result.
\end{proof}

Given the equivalences in \intref{Theorem}{thm:rob} one can recast the duality statement above in terms of the singularity category and the stable category of Gorenstein projective modules.   Here too we are following Buchweitz's footsteps~\cite{Buchweitz:1986}, except that he does not require $A$ to be projective over a central sub-algebra; on the other hand, he considers  only rings of finite injective dimension. We can get away with local finiteness of injective dimension, thanks to \intref{Theorem}{thm:Gconverse}.

\begin{corollary}
\label{cor:serre-sing}
For $R$ and $A$ as in \intref{Corollary}{cor:serre}, the singularity category  $\dsing(A)$ has Serre duality, with Serre functor $\Sigma^{d-1}\omega_{A/R}\lotimes_A(-)$. 
\end{corollary}

\begin{proof}
This is a direct translation of \intref{Corollary}{cor:serre}, made using \intref{Theorem}{prp:Nak-G-sing}.
\end{proof}

And here is \intref{Corollary}{cor:serre}  transported to the world of G-projective  modules. 

\begin{corollary}
\label{cor:serre-mcm}
For $R$ and $A$ as in \intref{Corollary}{cor:serre}, the functor 
\[
M \mapsto \Omega^{1-d}\gp(\omega_{A/R}\otimes_AM)
\]
is a Serre functor on the triangulated category $\uGproj A$. \qed
\end{corollary}

\begin{remark}
Set $S\colonequals \Omega^{1-d}\gp(\omega\otimes_A(-))$, the Serre functor on $\uGproj(A)$. \intref{Theorem}{cor:serre-mcm} translates to the statement that there is an $R$-linear \emph{trace} map
\[
\sHom_A(M,SM) \xra{\ \tau\ } I(\fm)
\]
such that the bilinear pairing 
\[
\sHom_A(N, SM) \times \sHom_A(M, N) \xra{\ -\circ-\ } \sHom_A(M, SM)) \xra{\ \tau\ } I(\fm)
\]
where $-\circ-$ is the obvious composition, is non-degenerate. Murfet~\cite{Murfet:2013} describes the trace map in the case when $A=R$, that is to say, in the case of commutative rings; this involves the theory of residues and differentials forms. It would be interesting to extend his work to the present context.
\end{remark}

We now prepare for the proof of \intref{Theorem}{thm:Gor}.

\begin{lemma}
\label{lem:rslocal}
Let $A$ be a finite $R$-algebra. For each $X$ in $\Loc(\ires A)$ for which $\ires X$ is in $\bfK^{+}(\Inj A)$, the isomorphism \eqref{lem:rs} induces isomorphisms
\[
\Sigma^{-1} \gam_{\fp}\bfs(\ires X)  \longiso \bfr \gam_{\fp}X \quad \text{for each $\fp\in\Spec R$.}
\]
\end{lemma}

\begin{proof}
Since $X$ is in $\Loc(\ires A)$, from \eqref{eq:rho} and \intref{Lemma}{lem:rs} we get an exact triangle
\[
\Sigma^{-1}\bfs(\ires X) \lra X \lra \ires X\lra 
\]
Applying $\gam_{\fp}$ to this yields the exact triangle
\[
\Sigma^{-1} \gam_{\fp} (\bfs(\ires X))  \lra \gam_{\fp} X \lra \gam_{\fp}(\ires X) \lra 
\]
Since $\bfs(\ires X)$ is acyclic so is the complex $\gam_{\fp} (\Sigma^{-1}\bfs(\ires X))$, by \intref{Lemma}{lem:lipman}. Hence $\bfr(-)$ is (isomorphic to) the identity on this complex.  On the other hand since $\ires X$ is bounded below, so is $\gam_{\fp}(\ires X)$, by \intref{Lemma}{lem:bb}, and hence $\bfr(-)$ vanishes on this complex. Keeping these observations in mind and applying the functor $\bfr$ to the exact triangle above yields the stated isomorphism.
\end{proof}

\begin{proof}[Proof of \intref{Theorem}{thm:Gor}]
It suffices to establish the result for objects of the form $\bfs(C)$ where $C\in \KInj A$ is a compact object; we can assume $C$ is bounded below. 
Set $D\colonequals \Hom_R(A,\ires R)$. We shall be interested in the complex of injective $A$-modules
\[
X\colonequals D \otimes_A \pres(\wh\kan C)\,.
\]
We claim that this complex satisfies the hypotheses of \intref{Lemma}{lem:rslocal}, namely

\begin{claim}
$X$ is in $\Loc(\ires A)$ and $\ires X\iso C$ and, in particular, it is bounded below.

\smallskip

Indeed,  the complex $D$ consists of $A$-bimodules that are injective on either side and the map $R\to \ires R$ induces a quasi-isomorphism 
\[
\omega=\Hom_R(A,R)\lra \Hom_R(A,\ires R) = D
\]
of $A$-bimodules. Thus $D$ is an injective resolution of $\omega$ on both sides. It follows that in $\dcat A$ there are natural isomorphisms  
\[
D\otimes_A\pres(\wh\kan C) \cong \omega\lotimes_A \RHom_A(\omega,C) \cong C
\]
where the second one is by \intref{Theorem}{thm:Morita-D}. Therefore in $\KInj A$ one gets that
\[
\ires X = \ires (D \otimes_A \pres(\wh\kan C))\longiso C. 
\]

As to the first part of the claim, $\pres(\wh\kan C)$ is in $\Loc(A)\subseteq \KProj A$, hence $X$ is in $\Loc(D)$ in $\KInj A$. However  $D$ is an injective resolution of $\omega$ and the latter is perfect, as an object of $\dcat A$, so $D$ is in $\Thick(\ires A)$. It follows that $X$ is in $\Loc(\ires A)$, as claimed.

\end{claim}

From the claim and \intref{Lemma}{lem:rslocal} we deduce that
\[
\Sigma^{-1}\gam_{\fp}\bfs(\ires X) \cong \bfr \gam_{\fp}X \quad \text{for each $\fp\in\Spec R$.}
\]
This justifies the penultimate isomorphism below, where $h$ stands for $\dim R_\fp$:
\begin{align*}
T_{\fp}(\wh\kan (\bfs C))
	&\cong T_{\fp}(\bfs(\wh\kan C)) \\
        & \cong \bfr(\Hom_R(A, I(\fp))  \otimes_A \pres(\wh\kan C)) ) \\
	& \cong \bfr (\Hom_R(A, \Sigma^{h} \gam_{\fp}{\ires R})  \otimes_A \pres(\wh\kan C))) \\
	&\cong \Sigma^{h}\bfr\gam_{\fp}(\Hom_R(A,\ires R)  \otimes_A \pres(\wh\kan C)) \\
	&= \Sigma^{h}\bfr\gam_{\fp}(X) \\
	&\cong \Sigma^{h-1} \gam_{\fp}\bfs(\ires X)\\
	&\cong \Sigma^{h-1} \gam_{\fp}\bfs(C)
\end{align*}
The first isomorphism is by \intref{Theorem}{thm:Morita-K}; the second is by \intref{Corollary}{cor:TsT}; the third is from \eqref{eq:gd-commutative} that applies as $R$ is Gorenstein, by \intref{Lemma}{lem:RGor}. The last isomorphism is again by the claim  above. This finishes the proof.
\end{proof}

In contrast with \intref{Proposition}{prp:Gor-D} and \intref{Proposition}{prp:Gor-K}, we do not know if the Gorenstein property of $\KacInj A$ characterises Gorenstein algebras; except when $A$ is commutative. 

\begin{theorem}
\label{thm:Kinj-gor}
Let $R$ be a commutative noetherian ring. The $R$-linear category $\KacInj R$ is Gorenstein if and only if the ring $R$ is Gorenstein.
\end{theorem}

\begin{proof}
The reverse implication is contained in \intref{Theorem}{thm:Gor}.   

Suppose $\KacInj R$ is Gorenstein as an $R$-linear category, with global Serre functor $F$. Let $\fm$ be a maximal ideal of $R$, and $k\colonequals R/\fm$, its residue field. The object $\bfs\ires k$ in $\KacInj A$ is compact and $\fm$-torsion so \intref{Proposition}{pr:Serre} yields
\begin{align*}
\KHom(\bfs\ires k, \bfs\ires k)^{\vee\vee}
	&\cong \KHom(\bfs\ires k, \Sigma^{d(\fm)}F(\bfs\ires k))^{\vee} \\
	&\cong \KHom(\Sigma^{d(\fm)} F(\bfs\ires k), \Sigma^{d(\fm)} F(\bfs\ires k)) \\
	&\cong \KHom(\bfs\ires k,\bfs\ires k) 
\end{align*}
Thus one gets an isomorphism of Tate cohomology modules
\[
\tExt^{0}_{R}(k,k) \cong \tExt^{0}_{R}(k,k)^{\vee\vee} \quad\text{for each $i\in\bbN$.}
\]
These modules are annihilated by $\fm$, so are $k$-vector spaces. The isomorphism above implies that each of them has finite rank over $k$. It remains to recall the result of Avramov and Veliche~\cite[Theorem~6.4]{Avramov/Veliche:2007} that the finiteness of the rank of $\tExt^{i}_{R}(k,k)$ for \emph{some} $i$ already implies $R_{\fm}$ is Gorenstein.
\end{proof}

The proof of the preceding result does not go through for non-commutative rings, for there do exist  finite dimensional algebras $A$ over a field $k$ that are not Gorenstein, and yet $\tExt^i_A(M,N)$ is finite dimensional over $k$ for each $i$, and finite dimensional $A$-modules $M,N$; see, for example, \cite[Example~4.3, (1), (2)]{Chen:2009}.

\begin{appendix}

\section{Gorenstein approximations}
\label{se:appendix}

Let $\cat A$ be an additive category. Recall that a complex $X\in\bfK(\cat A)$ is called \emph{totally acyclic} if the complexes of abelian groups $\Hom(W,X)$ and $\Hom(X,W)$ are acyclic  for all $W\in\cat A$. When $\cat A$ is abelian and $\cat C\subseteq\cat A$ is a class of objects, we set
\begin{align*}
^\perp\cat C &=\{X\in\cat A\mid\Ext^n(X,Y)=0\text{ for all }Y\in\cat C,\, n>0\} \\
\cat C^\perp & =\{Y\in\cat A\mid\Ext^n(X,Y)=0\text{ for all }X\in\cat C,\, n>0\}.
\end{align*}
A pair  $(\cat X,\cat Y)$ of full  subcategories of $\cat A$ is a (hereditary and complete) \emph{cotorsion pair} for $\cat A$ if
\[
\cat X^\perp =\cat Y\qquad\text{and}\qquad\cat X={^\perp\cat Y}
\]
and every object $M\in\cat A$ fits into exact \emph{approximation sequences}
\begin{equation*}
\label{eq:tilt-cotorsion}
0\lra Y_M\lra X_M\lra M\lra 0\qquad\text{and}\qquad 0\lra M\lra Y^M\lra X^M\lra 0
\end{equation*}
with $X_M,X^M\in\cat X$ and $Y_M,Y^M\in\cat Y$.

\subsection*{Gorenstein algebras}
Fix a ring $A$. Recall that an $A$-module is \emph{G-projective} if it is of the form 
 \[
 C^0(X)\colonequals \Coker (X^{-1}\xra{d^{-1}}X^0)
 \]
 for a  totally acyclic $X\in\bfK(\Prj A)$. The   \emph{G-injective} modules are those of the form
  \[
  Z^0(X) \colonequals \Ker (X^0\xra{d^0}X^1)
 \]  
for some totally acyclic $X\in\bfK(\Inj A)$. We write $\GProj A$ for the full subcategory of all G-projective modules and $\GInj A$ for the full subcategory of all G-injective modules The theorem below provides Gorenstein approximations for all modules over a Gorenstein algebra.

Let  $\Fin A$ be the full subcategory of $A$-modules having finite projective and finite injective dimension. When $A$ is a finite $R$-algebra we consider the category
\[
\Fin(A/R)\colonequals \{M\in \Mod A\mid M_\fp\in\Fin(A_\fp)\text{ for all } \fp\in\Spec R\}.
\]
Observe that when the $R$-algebra $A$ is Gorenstein $\Fin(A_\fp/R)$ is the category of $A_\fp$-modules of finite projective---equivalently, finite injective---dimension. One of the consequences of the result below is that, at least for Gorenstein algebras, $\Fin(A/R)$ is independent of the ring $R$.

\begin{theorem}
\label{thm:g-approximations}
  Let $A$ be a Gorenstein $R$-algebra. Then there are equalities
  \[
  (\GProj A)^\perp=\Fin(A/R) = {^\perp(\GInj A)}
\] 
Also $(\GProj A, \Fin(A/R))$ and $(\Fin(A/R),\GInj A)$ are cotorsion pairs for $\Mod A$.
\end{theorem}

The map $X_M\to M$ for $\cat X=\GProj A$ is called the \emph{G-projective  approximation}; we set $\gp(M)=X_M$. This module is unique up to morphisms that factor through a projective module. Analogously, the map $M\to Y^M$ for $\cat Y=\GInj A$ is called \emph{G-injective approximation} and we set $\gi(M)=Y^M$; it is unique up to morphisms that factor through an injective module.

\begin{proof}
  First observe that any acyclic complex of projective or injective $A$-modules is totally acyclic, by  \intref{Theorem}{thm:documenta}. This means that G-projective and  G-injective modules are obtained from acyclic complexes.

 We begin with the construction of G-injective approximations, using the recollement \eqref{eq:recollement} as follows. Set $\cat Y=\GInj A$ and $\cat X={^\perp\cat Y}$. Fix an $A$-module $M$. Then an injective resolution $\ires M$ fits into an exact triangle
\[
\bfj\bfq(\ires M)\lra \ires M\lra\bfs(\ires M)\lra
\]
given by an exact sequence of complexes 
\[
0\lra \ires M\lra\bfs(\ires M)\lra \Sigma(\bfj\bfq(\ires M))\lra 0
\] 
which is split-exact in each degree. Thus $Z^0(-)$  gives an exact sequence 
\[
0\lra M\lra Y^M\lra X^M\lra 0
\]
with $X^M\in\cat X$ and $Y^M\in\cat Y$. The other sequence $0\to Y_M\to X_M\to M\to 0$ is obtained by rotating this triangle. This justifies the claim that $(\cat X,\cat Y)$ is a cotorsion pair; see \cite[Theorem~7.12]{Krause:2005} for details.

It remains to identify $\cat X$, the left orthogonal to $\GInj A$.  A
standard argument yields the equality $\cat X=\Fin A$ when $A$ is
Iwanaga-Gorenstein. For a Gorenstein algebra $A$ the equality $\cat
X=\Fin (A/R)$ follows once we can show that for each $\fp\in\Spec R$
the $\fp$-localisation of an approximation sequence for $M\in\Mod A$
yields an approximation sequence for $M_\fp$ in $\Mod A_\fp$. It
follows from the discussion in \intref{Remark}{rem:s-local} that for any $A$-module $M$ one has isomorphisms
\[
(\bfs(\ires M))_\fp\cong \bfs_\fp((\ires M)_\fp)\cong
  \bfs_\fp(\ires_\fp M_\fp).
  \] 
This implies  $(X^M)_\fp\cong X^{M_\fp}$ and $(X_M)_\fp\cong X_{M_\fp}$. Thus both modules have finite projective and finite injective dimension. We conclude that $\cat X=\Fin (A/R)$.

Next we consider G-projective approximations using the analogue of the recollement \eqref{eq:recollement} for $\KProj A$. The proof that $(\GProj A, (\GProj A)^\perp)$ is a cotorsion pair is similar to the one for $\GInj A$, for it uses the right adjoint of the inclusion $\KacProj A\hookrightarrow\KProj A$; we omit the details. The equality $(\GProj A)^\perp=\Fin (A/R)$ can be verified as follows.  Recall from
\intref{Theorem}{thm:documenta} that there is an adjoint pair of
triangle equivalences
\[
\begin{tikzcd}
\KacProj A  \arrow[rightarrow,yshift=1ex]{rr}{E\otimes_A-}
 	&& \KacInj A \arrow[rightarrow,yshift=-1ex]{ll}{\bfh}[swap]{\sim}
\end{tikzcd}
\]
Consider the exact triangle
\[
E\otimes_A \pres M\lra \ires M\lra \bfs(\ires M)\lra
\]
from \intref{Lemma}{lem:lambda2} which we used for constructing G-injective approximation of $M$. Applying the equivalence $\bfh$ and rotating yields an exact triangle
\[
\Sigma^{-1}\bfh \bfs(\ires M)\lra \pres M\lra \bfh(\ires M) \lra
\]
which provides us with the G-projective approximation of
$M$. We claim that for each $\fp\in\Spec R$ the $\fp$-localisation of this
triangle yields the Gorenstein-projective approximation of $M_\fp$. To
this end consider the following diagram of exact functors.
\begin{equation*}
\begin{tikzcd}[row sep=large]
  \KProj A \arrow[yshift=.75ex]{rr}{E\otimes_A-}
  \arrow[twoheadrightarrow,xshift=-.75ex,swap]{d}{(-)_\fp}&& \KInj A
  \arrow[yshift=-.75ex]{ll}{\bfh}
  \arrow[twoheadrightarrow,xshift=-.75ex,swap]{d}{(-)_\fp}\\
 \KProj{A_\fp} \arrow[yshift=.75ex]{rr}{E_\fp\otimes_{A_\fp}-}
  \arrow[tail,xshift=.75ex,swap]{u}{\mathrm{res}}
 && \KInj{A_\fp}
  \arrow[yshift=-.75ex]{ll}{\bfh_\fp}
 \arrow[tail,xshift=.75ex,swap]{u}{\mathrm{res}}
\end{tikzcd}
\end{equation*}
It is easily checked that for each $A$-module $M$ one has isomorphisms
\[
(\bfh\bfs(\ires M))_\fp\cong \bfh_\fp((\bfs(\ires M))_\fp) \cong \bfh_\fp\bfs_\fp(\ires_\fp M_\fp).
  \] 
This implies that $(Y^M)_\fp\cong Y^{M_\fp}$ and $(Y_M)_\fp\cong Y_{M_\fp}$. Thus both modules have finite projective and finite injective dimension, so $(\GProj A)^\perp=\Fin (A/R)$.
\end{proof}

\begin{remark}
The above theorem  shows that Gorenstein algebras are virtually
Gorenstein in the sense of Beligiannis and Reiten
\cite{Beligiannis/Reiten:2007}, which means that the classes $(\GProj A)^\perp$
and $^\perp(\GInj A)$ coincide.
\end{remark}

\end{appendix}

\end{document}